\documentclass[a4paper,12pt,leqno]{amsart} 

\usepackage{amssymb, amsmath, amsthm}
\usepackage{amsfonts} 
\usepackage{color}
\usepackage{mathrsfs}
\usepackage{graphicx}
\usepackage{comment}

\numberwithin{equation}{section}
\numberwithin{equation}{section}
\newcommand{\e}{\varepsilon}
\newcommand{\R}{\mathbb R}
\newcommand{\T}{\mathbb{T}}
\newcommand{\dv}{{\rm{div}}}
\newcommand{\tenchi}{{}^t\!}
\newcommand{\wtts}{\overset{2,2}{\rightharpoonup}}
\newcommand{\vw}{\overset{2,2}{\underset{\rm vw}{\rightharpoonup}}}
\newcommand{\mx}{\tfrac{x}{\e}}
\newcommand{\mt}{\tfrac{t}{\e^r}}
\newcommand{\mnx}{\tfrac{x}{\e_n}}
\newcommand{\mnt}{\tfrac{t}{\e_n^r}}

\newtheorem{thm}{Theorem}[section]
\newtheorem{lem}[thm]{Lemma}
\newtheorem{rmk}[thm]{Remark}
\newtheorem{prop}[thm]{Proposition}
\newtheorem{defi}[thm]{Definition}

\newtheorem{corollary}[thm]{Corollary}

\allowdisplaybreaks[4]

\title[Space-time quasi-periodic homogenization for  damped wave equations]{Space-time arithmetic quasi-periodic homogenization\\ for damped wave equations} 
\author{Tomoyuki Oka}
\address[Tomoyuki Oka]{Graduate School of Sciences, Tohoku University, Sendai 980-8579 Japan}
\email{tomoyuki.oka.q3@dc.tohoku.ac.jp}

\date{\today}

\begin{document}
\maketitle

\begin{abstract}
This paper is concerned with space-time homogenization problems for damped wave equations with spatially periodic oscillating elliptic coefficients and temporally (arithmetic) quasi-periodic oscillating viscosity coefficients. Main results consist of a homogenization theorem, qualitative properties of homogenized matrices which appear in homogenized equations and a corrector result for gradients of solutions. In particular, homogenized equations and cell problems will turn out to deeply depend on the quasi-periodicity as well as the log ratio of spatial and temporal periods of the coefficients. Even types of equations will change depending on the log ratio and quasi-periodicity. Proofs of the main results are based on a (very weak) space-time two-scale convergence theory.
\end{abstract}

\section{Introduction and main results}

\emph{Space-time homogenization problems} for hyperbolic equations were first studied by Bensoussan, Lions and Papanicolaou. In \cite{BLP}, based on a method of \emph{asymptotic expansion}, the following wave equation is treated\/{\rm :}
\begin{equation}\label{wave}
\partial_{tt}^2u_{\e} -\dv\left(a_{\e}\nabla u_{\e} \right)=f \quad\text{ in } \Omega\times (0,T),
\end{equation} 
where $\Omega$ is a bounded domain in $\R^N$ with smooth boundary $\partial\Omega$, $N\ge1$, $T>0$, $f=f(x,t)$ is a given data, $a : \T^N \times \T\to \R^{N\times N}$ is an $N \times N$ symmetric matrix field satisfying a uniform ellipticity and $1$-periodicity and $a_{\e}:=a(\mx,\mt)$ for $r>0$ (i.e., $a_\e$ is $\e\times \e^r$-periodic). The homogenization problem concerns asymptotic behavior as $\e\to 0_+$ of (weak) solutions $u_\e=u_\e(x,t)$ as well as a rigorous derivation of limiting equations, often called \emph{homogenized equation}. In \cite{BLP}, it is assumed that (weak) solutions $u_\e = u_\e(x,t)$ can be expanded as a series\/{\rm :}
\begin{equation}\label{asym_exp}
u_\e(x,t) = \sum_{j = 0}^\infty \e^j u_j(x,t,\mx,\mt),
\end{equation}
where $u_j = u_j(x,t,y,s) : \Omega \times (0,T)\times \T^N\times \T \to \R$ for $j=0,1,2,\ldots$ are some periodic functions, and then, by substituting \eqref{asym_exp} to \eqref{wave}, at a formal level, $u_0=u_0(x,t)$ turns out to be independent of \emph{microscopic variable} $(y,s)$ and to solve the following homogenized equation\/:
\begin{equation*}\label{wave-hom}
\partial_{tt}^2u_{0} -\dv\left(a_{\rm hom}\nabla u_{0} \right)=f \quad\text{ in } \Omega\times (0,T),
\end{equation*} 
where $a_{\rm hom}$ is the so-called \emph{homogenized matrix} and represented as
\begin{equation}\label{hom_mat}
a_{\rm hom} e_k = \int_0^1\int_{\square} a(y,s)\bigl(\nabla_y \Phi_k(y,s)+e_k\bigl) \, dyds \quad \mbox{ for } \ k=1,2,\ldots,N.
\end{equation}
Here $\square := (0,1)^N$ is a unit cell, $\nabla_y$ stands for the gradient operator with respect to the third variable $y$, 
$\{e_k\} = \{[\delta_{jk}]_{j=1,2,\ldots,N}\}$ stands for a canonical basis of $\R^N$ and 
 $\Phi_k : \T^N\times \T \to \R$ (for $k = 1,2,\ldots,N$) is the \emph{corrector} which will be explained latter (see Remark \ref{corrector} below).
Moreover, $\Phi_k$ is determined by the so-called
\emph{cell problems}. 
In particular, if the log-ratio of the spatial and temporal periods of the coefficients is the hyperbolic scale ratio (i.e., $r=1$), then the cell problem is also a wave equation.
 \begin{equation*}
\partial_{ss}^2\Phi_k-\mathrm{div}_y \bigl[ a(y,s) (\nabla_y \Phi_k+ e_k ) \bigl] = 0 \quad \mbox{ in } \T^N\times \T
\end{equation*} 
(otherwise, cell problems are always elliptic equations, e.g.,~\eqref{CPslow} below).

In \cite{BLP}, the following heat equation is also treated\/{\rm :}
\begin{equation}\label{heat}
\partial_tu_{\e}-\dv\left(a_{\e}\nabla u_{\e} \right)=f \quad\text{ in } \Omega\times (0,T).
\end{equation}
By substituting \eqref{asym_exp} to \eqref{heat}, 
$u_0=u_0(x,t)$ is a (weak) solution to the following homogenized equation\/{\rm :}  
\begin{equation}\label{heat-hom}
\partial_tu_{0}-\dv\left(a_{\rm hom}\nabla u_{0} \right)=f \quad\text{ in } \Omega\times (0,T),
\end{equation}
where $a_{\rm hom}$ is defined by \eqref{hom_mat}. Furthermore, if $r=2$ (i.e., $a_\e=a(\mx,\tfrac{t}{\e^2})$), then the corrector $\Phi_k$ is the unique solution to the following cell problem\/{\rm :}
\begin{equation}\label{heat-CP}
\partial_{s}\Phi_k-\dv_y\bigl[a(y,s)(\nabla_y\Phi_k+e_{k})\bigl]=0\quad \text{ in }\ \T^N\times \T
\end{equation}
(as in \eqref{wave}, cell problems are always elliptic equations for any $r\neq 2$). Thus the type of the cell problem depends on the log-ratio of the spatial and temporal periods of the coefficients. Moreover, these formal arguments based on the asymptotic expansion for (the Cauchy-Dirichlet problem for) \eqref{heat} are justified via \emph{two-scale convergence theory} by A.~Holmbom in \cite{Ho}. The notion of two-scale convergence was first proposed by G.~Nguetseng~\cite{Ng}, and then, developed by G.~Allaire~\cite{Al1,Al2} (see also, e.g.,~\cite{LNW,Vi,Zh}). It enables us to analyze how strong compactness of bounded sequences in Sobolev spaces fails due to their oscillatory behaviors (see (i) and (ii) of Remark \ref{indepwtts} below). A.~Holmbom extended the two-scale convergence theory to space-time homogenization and derived \eqref{heat-hom} and \eqref{heat-CP} rigorously. Moreover, the notion of \emph{very weak two-scale convergence} is introduced, and then, it plays a crucial role for characterizing homogenized matrices (see Corollary \ref{veryweak} below for details). Besides, homogenization problems for various parabolic equations have been studied not only for linear ones but also for nonlinear ones (e.g.~\cite{AO,EP,FHOS,J,NW,W}). In particular, for p-Laplace type \cite{EP,W} and porous medium type \cite{AO}, it has been proved that cell problems are given as parabolic equations at the critical scale (i.e.,~$a_\e=a(\mx,\tfrac{t}{\e^2})$ in \eqref{heat}).

On the other hand, the following more general hyperbolic-parabolic equation is treated (e.g.~\cite{BFM,BL,CCMM, DT, FF,Mi,Ti2,To}). 
\begin{equation}\label{H-P}
h_{\e}\partial_{tt}^2u_{\e}-\dv(a_{\e}\nabla u_{\e})+g_{\e}\partial_{t}u_{\e}=f
\quad\text{ in } \Omega\times (0,T).
\end{equation}
Here $h_{\e}$ and $g_{\e}$ are $\e\times \e^r$-periodic functions rapidly oscillating. Furthermore, \cite{NNS,Nn,Ti1,WD} deal with nonlinear wave equations, and in particular, in \cite{NNS,Nn}, almost periodic settings are studied via $\Sigma$-convergence theory developed in \cite{Ng2}. Here \eqref{H-P} is called \emph{damped wave equations} for $h_\e\equiv1$ and $g_{\e}>0$ and it is noteworthy that asymptotic expansions of solutions to damped wave equations are performed with the aid of solutions to diffusion equations (e.g.~\eqref{heat}), and moreover, asymptotic behaviors of solutions to damped wave equations are similar to those of diffusion equations as $t\to +\infty$ (see e.g.~\cite{GR,N1}). Therefore, it is expected that cell problems for \eqref{H-P} will change at the critical scale for \eqref{heat} (i.e., $a_\e=a(\mx,\tfrac{t}{\e^2})$). However, at least to our knowledge, it does not seem to occur under the periodic homogenization in the fixed domain except for $h_\e=-\e^2$ and $g_\e\equiv 1$ (see \cite[Chapter 2, Section 4.5]{BLP} for details).

\subsection{Setting of the problem}
One of main purposes of the present paper is to find conditions under which the cell problems of \eqref{H-P} will be different from elliptic ones (see \eqref{CPslow} below). As a consequence, we emphasize that the (arithmetic) quasi-periodicity of the time-dependent coefficient $g_\e$ in \eqref{H-P} is crucial and it is defined as follows.
\begin{defi}[Quasi-periodic functions]\label{quasi}
The function $\varphi\in C(\R)$ is said to be \emph{(arithmetic) quasi-periodic} if it satisfies 
\begin{equation*}
\varphi(s+1)=\varphi(s)+C_{\ast}\quad \text{ for all $s\in [0,1)$ and \  $C_{\ast}\in \R$}
\end{equation*}
{\rm(}i.e.,~$\varphi$ is $(0,1)$-periodic if $C_{\ast}=0${\rm)}.
\end{defi}

\begin{rmk}
\rm
The notion of quasi-periodicity has been defined
in several different ways 
(see e.g.,~\cite{Co,Coo}). We stress that quasi-periodic functions in the sense of Definition \ref{quasi} do not satisfy the almost-periodicity in the sense of Besicovitch, which is known as a generalization of periodicity. Indeed, if $\varphi\in C(\R)$ is quasi-periodic, there exists a $(0,1)$-periodic function $\varphi_{\rm per}\in C_{\rm per}(\square)$
\footnote{Indeed, setting $\varPhi(s):= \varphi(s)-C_{\ast}s$, we see that $\varPhi(s)$ is $(0,1)$-periodic.}
such that
\begin{equation*}
\varphi(s)=\varphi_{\rm per}(s)+C_{\ast}s,
\end{equation*}
which implies that
$$
\left(\limsup_{R\to +\infty}\frac{1}{|2R|}\int_{-R}^{R}|\varphi(s)|^r\, ds\right)^{1/r}=+\infty,\quad \text{ for all }\ r\in [1,+\infty).
$$
Thus $\varphi$ does not belong to the generalized Besicovitch space $B^r(\R)$ (see e.g.,~\cite{CG, JKO}). Moreover, we shall consider both the effect of the periodic homogenization and the effect of the singular limit due to $\varphi(\mt)=\varphi_{\rm per}(\mt)+C_{\ast}\mt$ and $C_\ast\mt\to +\infty$ for $t>0$ as $\e\to 0_+$.
\end{rmk}

In this paper, we shall consider the Cauchy-Dirichlet problem for the following damped wave equation\/{\rm :} 
\begin{equation}\label{DW}
\displaystyle
\left\{
\begin{aligned}
&\partial_{tt}^2u_{\e}-\dv\bigl[ a\left(t,\tfrac{x}{\e}\right)\nabla u_{\e} \bigl]+g\left(\tfrac{t}{\e^r}\right)\partial_tu_{\e}=f_{\e} \quad\text{ in } \Omega\times (0,T), \\
&u_{\e}|_{\partial\Omega}=0,\quad 
u_{\e}|_{t=0}=v_\e^0,\quad
\partial_tu_{\e}|_{t=0}=v_\e^1.
\end{aligned}
\right.
\end{equation}
Here we make the following
\vspace{1mm}

\noindent
{\bf Assumption (A).}\
Let $\Omega$ be a bounded domain in $\R^N$ with smooth boundary $\partial\Omega$, $N\ge 1$.
\begin{itemize}
\item[(i)]
Let $T>0$, $\e>0$ and $r>0$. Let $v_\e^0\in H^1_0(\Omega)$ and $v_\e^1\in L^2(\Omega)$ be such that
\begin{align*}
v_\e^0\to v^0\quad \text{ weakly in } H^1_0(\Omega)\quad\text{ and }\quad 
v_\e^1\to v^1\quad \text{ weakly in } L^2(\Omega).
\end{align*}
Let $f_{\e}, f\in L^{2}(\Omega\times (0,T))$  be such that
$$
f_\e\to f\ \text{ weakly in }\ L^2(\Omega\times (0,T)).
$$
\item[(ii)]
The $N\times N$ symmetric matrix $a\in[C^1(0,T;L^{\infty}(\R^N))]^{N\times N}$ satisfies a uniform ellipticity, i.e.,~there exists $\lambda>0$ such that 
\begin{equation}
\label{ellip}
\lambda |\xi|^2\le a(t,y) \xi\cdot\xi\le |\xi|^2\ \text{ for any $\xi\in\R^N$ and a.e.~$(t,y)\in (0,T)\times \R^N$,}
\end{equation}
and $(0,1)^N$-periodicity\/{\rm :}
\begin{equation*}
a(t,y+e_j)=a(t,y)\quad \text{ a.e.~in $(t,y)\in (0,T)\times \R^N$}.
\end{equation*}
\item[(iii)]
Set $g\in C(\R;\R_+)$ as follows\/{\rm :} 
$$
g(s)=g_{\rm per}(s)+C_{\ast}s>0\  \text{ for all $s\in\R_+$}.
$$
Here $g_{\rm per}$ is a $(0,1)$-periodic function and $C_{\ast} \ge 0$ is a constant.
In addition, if $r=2$, we further assume $C_{\ast}\le \frac{2\lambda}{C_{\square}}$, where $C_{\square}=N/\pi^2$ is the best constant of the Poincar\'{e} inequality on the unit cell, that is,
$$
\|w\|_{L^2(\square)}\le C_{\square}\|\nabla w\|_{L^2(\square)}\quad \text{ for all }\ w\in H^1_{\rm per}(\square)
$$
 {\rm(}see Notation below{\rm)}.
\item[(iv)]
In addition, if $C_{\ast}\neq 0$ and $2<r<+\infty$, then
$a=a(y)$, $v_\e^0$, $v_\e^1$, $a(y)$, $g(s)$ and $f_\e$ are smooth, $(-\dv (a(\mx)\nabla v_\e^0))$, $(v_\e^1)$, 
$(f_\e)$ and $(\partial_t f_\e)$ are bounded in  $L^2(\Omega)$, $H^1_0(\Omega)$, $L^{\infty}(0,T;L^2(\Omega))$ and $L^2(\Omega\times (0,T))$, respectively. 
\end{itemize}

In this paper, we shall consider the convergence of solutions $(u_\e)$ to \eqref{DW} and the homogenized equation as $\e\to0_+$. We also discuss how the homogenized matrix can be represented for each $r>0$.

\subsection{Main results}
We start with the following definition of weak solutions to \eqref{DW}\/{\rm:}
\begin{defi}[Weak solution of \eqref{DW}]\label{sol}
A function $u_{\e}\in L^{\infty}(0,T;H^1_0(\Omega))$ is  said to be a weak solution to \eqref{DW}, if the following {\rm (i)-(iii)} are all satisfied\/{\rm:} 
\begin{itemize}
\item[(i)]{\rm(}Regularity{\rm)} $u_{\e}\in W^{2,2}(0,T;H^{-1}(\Omega))\cap W^{1,\infty}(0,T;L^2(\Omega))$.
\item[(ii)]{\rm(}Initial condition{\rm)} $u_{\e}(t)\to v^0_\e$ strongly in $L^{2}(\Omega)$ as $t\to0_+$ and $\partial_tu_{\e}(t)\to v^1_\e$ in $H^{-1}(\Omega)$ as $t\to0_+$.
\item[(iii)]{\rm(}Weak form{\rm)} It holds that, for all $\phi\in H^1_0(\Omega)$,
\begin{equation}\label{weakform}
\left\langle \partial_{tt}^2 u_{\e}(t),\phi\right\rangle_{H^1_0(\Omega)}
+A_{\e}^t(u_{\e}(t),\phi)
+\langle g(\tfrac{t}{\e^{r}})\partial_t u_{\e}(t),\phi\rangle_{H^1_0(\Omega)}	
=\langle f_{\e}(t),\phi\rangle_{H^1_0(\Omega)} 	
\end{equation}
for a.e.~in $t\in(0,T)$, where $A_{\e}^t(v,w)$ is a bilinear form in $H^1_0(\Omega)$ defined by
$$
A_{\e}^t(v,w) =\int_{\Omega}a\left(t,\tfrac{x}{\e}\right)\nabla v(x)\cdot \nabla w(x)\, dx\quad
\text{ for }\ v,w\in H^1_0(\Omega).
$$
\end{itemize}
\end{defi}

By Galerkin's method (cf.~{\cite[Theorem 12.2]{CD}}), we have
\begin{thm}[Existence and uniqueness of weak solutions to \eqref{DW}]\label{well-posedness}
Suppose that 
\begin{align*}
&a(t,\mx)\in [C^1(0,T;L^{\infty}(\Omega))]^{N\times N}_{\rm sym}, \
g(\mt)\in C(0,T),\
f_{\e}\in L^2(\Omega\times (0,T)),\\
&v^0_\e\in H^1_0(\Omega),\
v^1_\e\in L^2(\Omega) .
\end{align*}
Then for every $\e>0$ there exists a unique weak solution $u_{\e}$ to \eqref{DW}. 
\end{thm}

Then we first obtain the following homogenization theorem\/{\rm :}

\begin{thm}[Homogenization theorem]\label{HPthm}
Suppose that {\bf (A)} is satisfied. Let $u_{\e}\in L^{\infty}(0,T;H^1_0(\Omega))$ be a unique weak solution to \eqref{DW}. There exist $u_0\in L^{\infty}(0,T;H^1_0(\Omega))$ and $h\in L^2_{\rm loc}((0,T];H^{-1}(\Omega))$ such that, for any $\sigma>0$,  
\begin{align}
u_{\e}&\to  u_0 &&\text{weakly-$\ast$ in } L^{\infty}(0,T;H^1_0(\Omega)),\label{HPconv1}\\
u_{\e}&\to u_0 &&\text{strongly in } C([0,T];L^2(\Omega)),\label{HPconv2}\\
g(\mt)\partial_tu_{\e}&\to \langle g_{\rm per}\rangle_s\partial_t 
u_0+C_{\ast} h \quad &&\text{weakly in } 
\begin{cases}
L^2(0,T;H^{-1}(\Omega))  \text{ if } C_{\ast}=0, \\
L^2(\sigma,T;H^{-1}(\Omega)) \text{ if } C_{\ast}\neq 0,\\
\end{cases}\label{HPconv4}\\
\partial_{tt}^2u_{\e}&\to\partial_{tt}^2u_{0}&&\text{weakly in } 
\begin{cases}
L^2(0,T;H^{-1}(\Omega))  \text{ if } C_{\ast}=0, \\
L^2(\sigma,T;H^{-1}(\Omega)) \text{ if } C_{\ast}\neq 0,\\
\end{cases}\label{HPconv5}\\
a(t,\mx)\nabla u_{\e}&\to  \langle a(t)(\nabla u_0+\nabla_y u_1) \rangle_y &&\text{weakly in } [L^2(\Omega\times (0,T))]^N,\label{HPconv3}
\end{align}
where $\langle w\rangle_{s}=\int_{0}^1w (s)\, ds$, $\langle \hat{w}\rangle_{y}=\int_{\square}\hat{w}(y)\, dy$ and $u_1$ is written by 
\begin{align}
 u_1  (x,t,y):   =\sum_{k=1}^N\partial_{x_k}u_0(x,t)\Phi_k(t,y).\label{HPu1}
\end{align}
Here $\Phi_k$ is a corrector for each $k=1,\ldots, N$ and it is characterized as follows\/{\rm :}
\begin{itemize}
\item[{\rm(i)}] In case $r\in(0, +\infty)\setminus\{2\}$, $\Phi_k\in H^1_{\mathrm{per}}(\T^N)/\R$ {\rm(}see Notation below{\rm)} is the unique solution to  
\begin{equation}\label{CPslow}
-\dv_y\left[a(t,y)(\nabla_y\Phi_k+e_{k})\right]=0\ \text{ in }\ \T^N\times (0,T),
\end{equation}
where $e_{k}$ is the $k$-th vector of the canonical basis of $\R^N$.
\item
[{\rm (ii)}] In case $r=2$, $\Phi_k\in L^{2}(0,T;H^1_{\mathrm{per}}(\T^N)/\R)$ is the unique solution to
\begin{equation}\label{CPcritical}
C_{\ast}t\partial_t\Phi_k-\dv_y\left[a(t,y)(\nabla_y\Phi_k+e_{k})\right]=0\ \text{ in }\ \T^N\times (0,T).
\end{equation}
In particular, if either $C_{\ast}= 0$ or $a=a(y)$, then $\Phi_k\in H^1_{\mathrm{per}}(\T^N)/\R$ is the unique solution to \eqref{CPslow}. 
\end{itemize}
Furthermore, for any $C_{\ast}\ge 0$, $u_0$ is the unique weak solution to 
\begin{equation}\label{HDW2}
\displaystyle
\left\{
\begin{aligned}
&\partial_{tt}^2u_{0}-\dv\left[ a_{\rm hom}(t)\nabla u_{0} \right]+ \langle g_{\rm per}\rangle_s\partial_t u_0+C_{\ast} h  =f \quad\text{ in } \Omega\times (0,T), \\
&u_{0}|_{\partial\Omega}=0 , \quad 	u_{0}|_{t=0}=v^0,\quad 
\partial_tu_{0}|_{t=0}= \tilde{v}^1. 
\end{aligned}
\right.
\end{equation} 
Here $u_0\equiv v^0$ whenever $C_{\ast}\neq 0$,  and moreover, 
\begin{equation*}\label{v1tilde}
\tilde{v}^1=
\begin{cases}
v^1 &\text{ if }\ C_{\ast}=0,\\
0 &\text{ if }\ C_{\ast}\neq 0.
\end{cases}
\end{equation*}
Moreover, $a_{\rm hom}(t)$ is the homogenized matrix given by
\begin{equation}\label{a_hom}
a_{\rm hom}(t)e_k=\int_{\square}a(t,y)\bigl(\nabla_y \Phi_k(t,y)+e_k\bigl)\, dy, \quad k=1,2,\ldots,N.
\end{equation}
\end{thm}

\begin{rmk}
\rm
It is noteworthy that, due to the loss of the time periodicity, the following facts hold\/{\rm :} 
\begin{itemize}
\item[(i)] {\bf(Homogenized equation).\/}
The homogenized equation \eqref{HDW2} is of the same type as the original equation \eqref{DW} for the periodic case (i.e., $C_{\ast}= 0$). On the other hand, for the quasi-periodic case (i.e., $C_{\ast}\neq 0$), by the effect of the singular limit of $g$, \eqref{HDW2} is represented as the following elliptic equation\/{\rm :}
\begin{equation*}
-\dv (a_{\rm hom}\nabla u_{0}) =f-C_{\ast}h \ \text{ in }\ \Omega\times (0,T), \quad u_0\in H^1_0(\Omega).
\end{equation*}
Furthermore, the limit of the solution to \eqref{DW} coincides with the limit of the initial data $v_\e^0$.
\item[(ii)] {\bf(Cell problem).\/}
For the periodic case $C_{\ast}= 0$, the corrector $\Phi_k$ is always described as the solution to the elliptic equation \eqref{CPslow}. On the other hand, for the quasi-periodic case, at the critical case $r=2$, the cell problem \eqref{CPcritical} is different from \eqref{CPslow} and it is given as the parabolic equation by the effect of the singular limit of $g$.
Thus $\Phi_k$ depends on $t\in (0,T)$, and then, 
qualitative properties of the homogenized matrix $a_{\rm hom}$ will change due to \eqref{a_hom} (see Proposition \ref{property_of_a_hom} below).

\end{itemize}
\end{rmk}

Moreover, as for the homogenized matrix, we next have the following 
\begin{prop}[Qualitative properties of the homogenized matrix $a_{\rm hom}$]\label{property_of_a_hom}
Under the same assumption as in Theorem \ref{HPthm},
let $0<r<+\infty$ and $a_{\rm hom} (t)$ be the homogenized matrices defined by \eqref{a_hom}. 
Then the following {\rm (i)} and {\rm(ii)} hold\/{\rm:}
\begin{itemize}
\item[(i)]{\rm(}Uniform ellipticity{\rm)}
It holds that
\begin{align*}
&\lambda|\xi|^2+ \lambda\|\nabla_y\Phi_{\xi}(t)\|^2_{L^2(\square)}+\frac{C_{\ast}t}{2}\frac{d}{dt}\|\Phi_{\xi}(t)\|_{L^2(\square)}^2\\
&\le  
a_{\rm hom}(t)\xi\cdot \xi 
\le 
|\xi|^2+ \|\nabla_y\Phi_{\xi}(t)\|^2_{L^2(\square)}+\frac{C_{\ast}t}{2}\frac{d}{dt}\|\Phi_{\xi}(t)\|_{L^2(\square)}^2 
\end{align*}
for any $\xi\in \R^N$ and a.e. $t\in(0,T)$, where $\lambda > 0$ is the ellipticity constant of $a(t,y)$ defined by \eqref{ellip} and $\Phi_{\xi}$ is the corrector given by either \eqref{CPslow} or \eqref{CPcritical} with $e_k$ replaced by $\xi\in \R^N$.\vspace{3mm}
\item[(ii)]{\rm(}Symmetry and asymmetry{\rm)}
If $a(t,y)$ is the symmetric matrix, then $a_{\rm hom}(t)$ is the asymmetric matrix for $r=2$ and $C_{\ast}\neq 0$. Otherwise, $a_{\rm hom}(t)$ is also the symmetric matrix.  
\end{itemize}
\end{prop}

\begin{rmk}
\rm
We stress that, in the critical case (i.e., $r=2$ and $C_{\ast}\neq 0$), even though the elliptic constant of $a(t, y)$ is independent of $t \in (0, T)$, that of $a_{\rm hom}(t)$ depends on $t$.
Furthermore, the symmetry breaking of $a_{\rm hom}(t)$ occurs 
but it makes no contribution to the divergence (see Remark \ref{skew} below). 
\end{rmk}

We finally get the following corrector result. 
\begin{thm}[Corrector result for time independent coefficients]\label{CR}
Suppose that {\bf(A)} is fulfilled and assume that 
$a=a(y)$, $v_\e^0$, $v_\e^1$, $a(y)$, $g(s)$ and $f_\e$ are smooth, $(-\dv (a(\mx)\nabla v_\e^0))$, $(v_\e^1)$, 
$(f_\e)$ and $(\partial_t f_\e)$ are bounded in  $L^2(\Omega)$, $H^1_0(\Omega)$, $L^{\infty}(0,T;L^2(\Omega))$ and $L^2(\Omega\times (0,T))$, respectively. 
Let $u_{\e}$ and $u_0$ be the unique solutions to \eqref{DW} and \eqref{HDW2}, respectively. Then it holds that 
\begin{equation}\label{errorest}
\displaystyle \lim_{\e\to 0_+}	\int_{0}^T\int_{\Omega} \left|\nabla u_{\e}(x,t)-\bigl(\nabla u_0(x,t)+\nabla_y u_1(x,t,\mx)\bigl)\right|^2\, dxdt=0 
\end{equation}
for all $r\in (0,+\infty)$, where $u_1=\sum_{k=1}^N\partial_{x_k}u_0\Phi_k$ and $\Phi_{k}\in L^2(0,T;H^1_{\rm per} (\T^N)/\R)$ is the corrector for $r\in(0,+\infty)$. 
\end{thm}

As for the time dependent case $a=a(t,y)$, we have the following corrector result for more specific settings\/{\rm :}
\begin{corollary}[Corrector result for time dependent coefficients]\label{CR2}
Suppose that $C_{\ast}\neq 0$. In addition, assume that $a(t,y)$ is smooth and the following \eqref{ellip2}-\eqref{f-add} hold\/{\rm :}
\begin{align}
&\partial_t a(t,y)\xi\cdot\xi \le 0\ 
\quad \text{ for all $\xi\in \R^N$ and all $(t,y)\in (0,T)\times \R^N$}, \label{ellip2} \\
&-\dv(a(0,\mx)\nabla v_\e^0)\to -\dv(a_{\rm hom}(0)\nabla v^0) \text{ strongly in $H^{-1}(\Omega)$}, \label{strong hinv}\\
&\lim_{\e\to 0_+}\|v_\e^1\|_{L^2(\Omega)}=0, \label{initialv-add}\\
&f_\e\to f \text{ strongly in } L^2(\Omega\times (0,T))
\text{ or } (f_\e/\sqrt{t})  \text{ is bounded in } L^2(\Omega\times (0,T)),
\label{f-add}
\end{align} 
In addition, if $r=2$ and $C_\ast \neq 0$, assume that
\begin{equation}
\partial_t a(t,y)=-a(t,y) \quad \text{ for all $(t,y)\in (0,T)\times \R^N$}. \label{ellip3}
\end{equation}
Here $a_{\rm hom}(t)$ is the homogenized matrix defined by \eqref{a_hom}.
Let $u_{\e}$ and $u_0$ be the unique solutions to \eqref{DW} and \eqref{HDW2}, respectively. Then \eqref{errorest} holds. 
\end{corollary}

\begin{rmk}	
\rm
Initial data $v_\e^0\in H^1_0(\Omega)$ satisfying \eqref{strong hinv} can actually be constructed (see e.g.~\cite[pp.~236]{CD}).
\end{rmk}

\begin{rmk}\label{corrector}
\rm
From Theorem \ref{CR}, it holds that 
\begin{equation*}
u_{\e}\not\to u_0 \quad \text{ strongly in }L^2(0,T;H^1_0(\Omega))
\end{equation*}
in general due to the oscillation of the third term $u_1(x,t,\mx)$ as $\e\to 0_+$.
Thus $u_1(x,t,\mx)$ plays a role as the corrector term recovering the strong compactness in this topology. For this reason, $\Phi_k$ is often called a corrector. 
\end{rmk}

\subsection{Plan of the paper and notation}
This paper is organized as follows. In the next section, we summarize relevant material on space-time two-scale convergence. Section $3$ is devoted to proving uniform estimates for solutions $u_\e$ to \eqref{DW} as $\e\to 0_+$. Furthermore, we shall prove their weak(-$\ast$) and strong convergences.~In Section $4$, we shall prove Theorem \ref{HPthm}. To prove Proposition \ref{property_of_a_hom}, we shall discuss qualitative properties of the homogenized matrix $a_{\rm hom}(t)$ in Section 5. The final section is devoted to proofs of Theorem \ref{CR} and Corollary \ref{CR2}.

\noindent
{\bf Notation.}\ 
Throughout this paper, $C>0$ denotes a non-negative constant which may vary from line to line. In addition, the subscript A of $C_{A}$ means dependence of $C_{A}$ on $A$. Let $\delta_{ij}$ be the Kronecker delta, $e_i=(\delta_{ij})_{1\le j\le N}$ be the $i$-th vector of the basis of $\R^N$, $\|\cdot\|_{H^1_0(A)}$ be defined by $\|\cdot\|_{H^1_0(A)}:=\|\nabla\cdot\|_{L^2(A)}$ for domains $A\subset \R^N$, $\nabla$ and $\nabla_y$ denote gradient operators with respect to $x$ and $y$, respectively, and $\dv$ and $\dv_y$ denote divergence operators with respect to $x$ and $y$, respectively. Furthermore, we shall use the following notation\/:
\begin{itemize}
\item $\square=(0,1)^N$, \quad $I=(0,T)$,\quad $J=(0,1)$,\quad $dZ=dydsdxdt$.
\item Define the set of smooth $\square$-periodic functions by
\begin{align*}
C^{\infty}_{\rm per}(\square) 
&= \{w\in C^{\infty}(\square) \colon w(\cdot+e_k)=w(\cdot) \text{ in } \R^N \ \text{ for }\ 1\leq k \leq N\}.
\end{align*}
\item We also define $W^{1,q}_{\rm per}(\square)$ and $L^q_{\rm per}(\square)$ as closed subspaces of $W^{1,q}(\square)$ and $L^q(\square)$ by
$$
W^{1,q}_{\rm per}(\square) = \overline{C^\infty_{\rm per}(\square)}^{\|\cdot\|_{W^{1,q}(\square)}}, \quad L^q_{\rm per}(\square) = \overline{C^\infty_{\rm per}(\square)}^{\|\cdot\|_{L^q(\square)}},
$$ 
respectively, for $1\leq q < +\infty$. In particular, set $H^1_{\rm per}(\square) := W^{1,2}_{\rm per}(\square)$. We shall simply write $L^q(\square)$ instead of $L^q_{\rm per}(\square)$, unless any confusion may arise.
\item We often write $L^q(\Omega\times \square)$ instead by $L^q(\Omega;L^q_{\rm per}(\square))$ since $L^q_{\rm per}(\square)$ is reflexive Banach space for $1<q<+\infty$. 
\item Define the mean $\langle w \rangle_y := \int_{\square} w(y) \, d y$ in $y$ of $w \in L^1(\square)$. 
\item We set $W^{1,q}_{\mathrm{per}}(\square)/\R=\{w\in W^{1,q}_{\mathrm{per}}(\square)\colon \langle w\rangle_{y}=\int_{\square}w(y)\, dy =0\}$.
\item Furthermore, let $X$ be a normed space  with a norm $\|\cdot\|_X$ and a duality pairing $\langle \cdot, \cdot \rangle_X$ between $X$ and its dual space $X^*$. Moreover, we write $X^N = X \times X \times \cdots \times X$ ($N$-product space), e.g., $[L^2(\Omega)]^N = L^2(\Omega;\R^N)$.
\end{itemize}

In order to clarify variables of integration, we shall often write, e.g., 
$\|u(\tfrac{x}{\e})\|_{L^{q}(\Omega)}$ and $\|u(x,\tfrac{x}{\e})\|_{L^{q}(\Omega)}$ instead of $\|u(\tfrac{\cdot}{\e})\|_{L^{q}(\Omega)}$ and $\|u(x,\tfrac{\cdot}{\e})\|_{L^{q}(\Omega)}$, respectively.
We often write $u(t)$ instead of $u(\cdot,t)$ for each $t\in I$ and $u:\Omega\times I\to\R$.

\section{Space-time two-scale convergence theory}
In this section, we introduce the notion of \emph{space-time two-scale convergence} and briefly summarize its crucial properties (see, e.g., \cite{Ng}, \cite{Al1,Al2}, \cite{LNW}, \cite{Zh} and \cite{Ho} for more details).
Throughout this section, let $q\in [1,+\infty]$ and $q'$ denote for the H\"older conjugate of $q$ (i.e., $1/q+1/q'=1$ if $1<q<+\infty$; $q'=1$ if $q=+\infty$; $q'=+\infty$ if $q=1$) unless any confusion may arise.
Furthermore, let $I=(0,T)$, $\square=(0,1)^N$ and $J=(0,1)$.

We first define a class of test functions for the space-time two-scale convergence, called \emph{admissible test functions}.
\begin{defi}[Admissible test function]\label{admissible}
Let $q'\in [1,+\infty]$ and let $X\subset L^{q'}(\Omega\times I\times \square\times J)$ be a separable normed space equipped with norm $\|\cdot\|_X$.
Then $(X,\|\cdot\|_X)$ is called an \emph{admissible test function space} {\rm (}for the weak space-time two-scale convergence in $L^q(\Omega\times I\times \square \times J)${\rm)}, if it holds that, for all $\Psi \in X$, $(x,t)\mapsto \Psi(x,t,\mx,\mt)$ is Lebesgue measurable in $\Omega\times I$ for $\e>0$, and
\begin{align*}
&\lim_{\e\to 0_+}\left\|\Psi\left(x,t,\tfrac{x}{\e},\tfrac{t}{\e^r}\right)\right\|_{L^{q'}(\Omega\times I)}  =  \left\|\Psi(x,t,y,s)\right\|_{L^{q'}(\Omega\times I\times \square\times J)}, 
\\
 &\left\|\Psi\left(x,t,\tfrac{x}{\e},\tfrac{t}{\e^r}\right)\right\|_{L^{q'}(\Omega\times I)} \le  C\left\|\Psi(x,t,y,s)\right\|_{X}\quad \text{ for }\ \e>0. 
\end{align*}
Moreover, $\Psi\in$ $X$ is called an \emph{admissible test function} {\rm (}for the weak space-time two-scale convergence in $L^q(\Omega\times I\times \square \times J)${\rm)}.
\end{defi}

The following fact is well known and often used, in particular, to discuss weak convergence of periodic test functions.
\begin{prop}[Mean-value property]\label{mean}
Let $w\in L^q(\square\times J)$ and set $w_{\e}(x,t)=w(\mx,\mt)$ for $\e>0$ and $0<r<+\infty$. 
For any bounded domain $\Omega\subset\R^N$ and any bounded interval $I\subset \R$, it holds that 
\begin{align*}
\begin{cases}
w_{\e}\to \langle w\rangle_{y,s} \ \text{ weakly in } L^q(\Omega\times I) \quad &\text{ if }\ q\in[1,+\infty), \\
w_{\e}\to \langle w\rangle_{y,s}\ \text{ weakly-}\ast\text{ in } L^{\infty}(\Omega\times I)
\quad &\text{ if }\ q=+\infty \\
\end{cases}
\end{align*}
as $\e\to 0_+$. Here $\langle w\rangle_{y,s}$ denotes the mean of $w$, i.e., 
$$
\langle w\rangle_{y,s}=
\int_0^1\int_{\square}w(y,s)\, dyds.
$$
\end{prop}
\begin{proof}
See \cite[Theorem 2.6]{CD}. 
\end{proof}
Now, we are in a position to define the notion of \emph{space-time two-scale convergence} in the following.
\begin{defi}[Weak space-time two-scale convergence and very weak two-scale convergence]\label{wtts}\ 
\begin{itemize}
\item[(i)]
A bounded sequence $(v_{\e})$ in $L^q(\Omega\times I)$ is said to \emph{weakly space-time two-scale converge} to a limit $v\in L^q(\Omega\times I\times \square\times J)$ if it holds that
\begin{align*}
\lim_{\e\to 0_+}
\int_0^T\int_{\Omega}v_{\e}(x,t)\Psi(x,t,\tfrac{x}{\e},\tfrac{t}{\e^r})\, dxdt
=
\int_0^T\int_{\Omega}\int_{0}^1\int_{\square}v(x,t,y,s)\Psi(x,t,y,s)\, dZ
\end{align*}
for any admissible function $\Psi\in X\subset L^q(\Omega\times I\times \square\times J)$ and it is denoted by
\[
v_{\e}\wtts v \quad \text{ in }\ L^q(\Omega\times I\times \square\times J).
\]
\item[(ii)]
A bounded sequence $(v_{\e})$ in $L^q(\Omega\times I)$ is said to \emph{very weakly two-scale converge to a limit $v$ in $L^q(\Omega\times I\times \square\times J)$} when we choose $\Psi(x,y,t,s)=\phi(x)b(y)\psi(t)c(s)$ for any $\phi\in C^{\infty}_{c}(\Omega)$, $b\in C^{\infty}_{\rm per}(\square)/\R$, $\psi\in C^{\infty}_c(I)$ and $c\in C^{\infty}_{\rm per}(J)$, and then, 
it is written by
\[
v_{\e}\vw v \quad \text{ in }\ L^q(\Omega\times I\times \square\times J).
\]
\end{itemize}
\end{defi}

\begin{rmk}
\rm
Due to the extension of the original definition in \cite{Ng, Al1} of the test function to $\Psi\in L^{q'}(\Omega\times I;C_{\rm per}(\square\times J))$ the boundedness of $(v_{\e})$ is essential. Indeed, some counterexamples that the (weak space-time) two-scale limit does not coincide with the weak limit are known in \cite[Examples 11 and 12]{LNW}.
\end{rmk}
 
\begin{rmk}\label{indepwtts}
\rm 
As for the relation between weak or strong convergence and weak space-time two-scale convergence, the following holds\/{\rm :}
\begin{itemize}
\item[(i)] If $v_{\e}\wtts v$ in $L^{q}(\Omega\times I\times \square\times J)$, then $v_{\e}\to \langle v\rangle_{y,s}$ weakly in $L^q(\Omega\times I)$.
\item[(ii)] If $v_{\e}\to \hat{v}$ strongly in $L^{1}(\Omega\times I)$, then $v_{\e}\wtts  \hat{v}$ in $L^{q}(\Omega\times I\times \square\times J)$. 
\end{itemize}
\end{rmk}
The following theorem is concerned with weak space-time two-scale compactness of bounded sequences in $L^q(\Omega\times I)$.

\begin{thm}[Weak space-time two-scale compactness]\label{multicpt}
Let $q\in (1,\infty]$. Then, for any bounded sequence $(v_{\e})$ in $L^q(\Omega\times I)$, there exist a subsequence $(\e_k)$ of $(\e)$ such that $\e_k\to 0_+$ and a limit $v\in L^q(\Omega\times I\times \square\times J)$ such that 
$$
v_{\e_k}\wtts v\quad \text{ in }\ L^q(\Omega\times I\times \square\times J).
$$
\end{thm}

\begin{proof}
See \cite[Theorem 2.3]{Ho}.
\end{proof}

As for weak space-time two-scale compactness of gradients, we obtain
\begin{thm}[Weak space-time two-scale compactness for gradients]\label{gradientcpt}\ 
Let $q\in (1,+\infty)$ and let $(v_{\e})$ be a bounded sequence in $W^{1,q}(\Omega\times I)$.
Then there exist a subsequence $\e_k\to 0_+$, a limit $v\in L^q(\Omega\times I)$ and a function $v_1\in L^q(\Omega\times I; W^{1,q}_{\rm per}(\square\times J)/\R)$ such that 
\begin{align*}
\nabla_{t,x} v_{\e_k}\wtts \nabla_{t,x} v+\nabla_{s,y} v_1\quad \text{ in }\ [L^q(\Omega\times I\times \square\times J)]^{N+1}.
\end{align*}
Here and henceforth, $\nabla_{t,x}=(\partial_t,\partial_{x_1},\ldots,\partial_{x_N})$ and $\nabla_{s,y}=(\partial_s, \partial_{y_1},\ldots,\partial_{y_N})$. 
\end{thm}

\begin{proof}
See \cite[Theorem 20]{LNW}. 
\end{proof}

As a corollary of Theorem \ref{gradientcpt}, the following is obtained.     
\begin{corollary}[cf.~\cite{FHO,FHOP,Ho}]\label{veryweak}
Under the same assumptions as in Theorem \ref{gradientcpt}, 
it holds that
\begin{equation*}
\frac{v_{\e_k}}{\e_k}\vw v_1\quad\text{ in }\ L^q(\Omega\times I\times \square\times J).
\end{equation*}
\end{corollary}

\begin{proof}
See \cite[Corollary 3.3]{Ho}. 
\end{proof}

\section{Uniform estimates}
To discuss convergence of solutions $u_{\e}$ for \eqref{DW}, we shall first verify their uniform boundedness, and then, we shall prove their weak or strong convergences.
\begin{lem}[Uniform estimates]\label{bdd}
Let $u_{\e}\in L^\infty(I;H^1_0(\Omega))$ be the unique weak solution of \eqref{DW}
under the same assumptions as in Theorem \ref{HPthm} and let $I_{\sigma}=(\sigma, T)$ for any $\sigma>0$. Then the following {\rm(i)-(vi)} hold\/{\rm :}
\begin{enumerate}
\rm
\item[(i)]
$(u_{\e}) $ is bounded in $L^{\infty}(I;H^1_{0}(\Omega))$,
\item[(ii)]
$(\partial_tu_{\e})$ is bounded in $L^{\infty}(I;L^{2}(\Omega))$,
\item[(iii)]
$(\sqrt{t\e^{-r}}\partial_{t}u_{\e})$ is bounded in $L^2(\Omega\times I)$, provided that $C_{\ast}\neq 0$,
\item[(iv)]
$(\partial_{tt}^2u_{\e}+g(\mt)\partial_{t}u_{\e})$ is bounded in $L^2(I;H^{-1}(\Omega))$,
\item[(v)]
$(\partial_{tt}^2u_{\e})$ is bounded in 
$\begin{cases}
L^2(I;H^{-1}(\Omega)) &\text{ if } C_{\ast}= 0,\\
L^2(I_\sigma;H^{-1}(\Omega)) &\text{ if } C_{\ast}\neq 0,
\end{cases}$
\item[(vi)]
$(t\e^{-r}\partial_{t}u_{\e})$ is bounded in $L^2(I_\sigma;H^{-1}(\Omega))$, provided that $C_{\ast}\neq 0$.
\end{enumerate}
\end{lem}

\begin{proof}
Recall \eqref{weakform}, i.e., 
\begin{align*}
\langle \partial_{tt}^2u_{\e}(t), \phi\rangle_{H^1_0(\Omega)}+A_{\e}^t(u_{\e}(t),\phi)+\langle g(\mt)\partial_t u_{\e}(t), \phi\rangle_{H^1_0(\Omega)}
=\langle f_{\e}(t), \phi\rangle_{H^1_0(\Omega)}
\end{align*}
for all $\phi\in H^1_0(\Omega)$.
Testing it by $\partial_t u_{\e}$ (see Remark \ref{test reg} below), we deduce by the symmetry of $a(t,y)$ that
\begin{align}
\lefteqn{\int_{\Omega}a(t,\mx)\nabla u_{\e}(x,t)\cdot \nabla \partial_t u_{\e}(x,t)\, dx}\label{Herm-mainterm}\\
&=
\frac{1}{2}\frac{d}{dt}\int_{\Omega}a(t,\mx)\nabla u_{\e}(x,t)\cdot \nabla u_{\e}(x,t)\, dx
-
\frac{1}{2}\int_{\Omega}\partial_t a(t,\mx)\nabla u_{\e}(x,t)\cdot \nabla u_{\e}(x,t)\, dx\nonumber
\end{align}
a.e.~in $I$. Thus we have
\begin{align}
\lefteqn{\frac{1}{2}\int_0^s\frac{d}{dt}\int_{\Omega}\Bigl[|\partial_t u_{\e}(x,t)|^2+a\left(t,\mx\right)\nabla u_{\e}(x,t)\cdot \nabla u_{\e}(x,t)\Bigl]\, dxdt}\label{re-weak}\\
&\stackrel{\eqref{Herm-mainterm}}{=}
\frac{1}{2}\int_0^s\int_{\Omega}\partial_t a(t,\mx)\nabla u_{\e}(x,t)\cdot \nabla u_{\e}(x,t)\, dxdt\nonumber
\\
&\quad+
\int_0^s\int_{\Omega}f_{\e}(x,t) \partial_tu_{\e}(x,t)\, dxdt-\int_0^s\Bigl(g_{\rm per}(\tfrac{t}{\e^r})+C_{\ast}
\tfrac{t}{\e^r}\Bigl)\|\partial_tu_{\e}(t)\|_{L^2(\Omega)}^2\, dt\nonumber
\end{align}
for all $s\in I$.
Then we observe from the uniform ellipticity \eqref{ellip}, \eqref{re-weak} and {\bf (A)} that  
\begin{align}
&\|\partial_t u_{\e}(s)\|_{L^2(\Omega)}^2+\lambda\|u_{\e}(s)\|_{H^1_0(\Omega)}^2\nonumber\\
&\stackrel{\eqref{ellip}}{\le}\|v^1_\e\|_{L^2(\Omega)}^2+\|v^0_\e\|_{H^1_0(\Omega)}^2+\int_0^s\frac{d}{dt}\left(\|\partial_t u_{\e}(t)\|_{L^2(\Omega)}^2+\int_{\Omega}a(t,\mx)\nabla u_{\e}(x,t)\cdot\nabla u_{\e}(x,t)\, dx\right)\, dt\nonumber\\
&\stackrel{\eqref{re-weak}}{=}\|v^1_\e\|_{L^2(\Omega)}^2+\|v^0_\e\|_{H^1_0(\Omega)}^2
+\int_0^s\int_{\Omega}\partial_t a(t,\mx)\nabla u_{\e}(x,t)\cdot \nabla u_{\e}(x,t)\, dxdt
\nonumber\\
&\quad+2\left( \int_0^s\int_{\Omega}f_{\e}(x,t) \partial_t u_{\e}(x,t)\, dxdt-\int_0^s\Bigl(g_{\rm per}(\mt)+C_{\ast}\tfrac{t}{\e^r}\Bigl)\|\partial_tu_{\e}(t)\|_{L^2(\Omega)}^2\, dt\right)\nonumber\\
&\stackrel{{\bf(A)}}{\le} \|v^1_\e\|_{L^2(\Omega)}^2+\|v^0_\e\|_{H^1_0(\Omega)}^2+\sup_{t\in I}\|\partial_t a(t)\|_{L^\infty(\square)}\int_0^s\|u_\e(t)\|_{H^1_0(\Omega)}^2\, dt\nonumber\\
&\quad+2 \int_0^s\left[\|f_{\e}(t)\|_{L^2(\Omega)}\| \partial_t u_{\e}(t)\|_{L^2(\Omega)}+\Bigl(\beta-C_{\ast}\tfrac{ t}{\e^r}\Bigl)\|\partial_tu_{\e}(t)\|_{L^2(\Omega)}^2\right]\, dt\nonumber\\
&\le\left(\|v^1_\e\|_{L^2(\Omega)}^2+\|v^0_\e\|_{H^1_0(\Omega)}^2+\|f_{\e}\|_{L^2(\Omega\times I)}^2\right)\nonumber\\
&\quad +C_{\beta}\int_0^s\left(\|\partial_tu_{\e}(t)\|_{L^2(\Omega)}^2+\|u_\e(t)\|_{H^1_0(\Omega)}^2\right)\, dt-C_\ast\int_0^s\|\sqrt{t\e^{-r}}\partial_tu_{\e}(t)\|_{L^2(\Omega)}^2\, dt.\nonumber
\end{align}
Here $\beta=\max_{s\in [0,1]}|g_{\rm per}(s)|$. From the boundedness of $(f_{\e})$ in $L^2(\Omega\times I)$, we get
\begin{align}
&\|\partial_t u_{\e}(s)\|_{L^2(\Omega)}^2+\lambda\|u_{\e}(s)\|_{H^1_0(\Omega)}^2+C_\ast\int_0^s\|\sqrt{t\e^{-r}}\partial_tu_{\e}(t)\|_{L^2(\Omega)}^2\, dt\label{bdd1}\\
&\quad \le C+C_{\beta}\int_0^s\left(\|\partial_tu_{\e}(t)\|_{L^2(\Omega)}^2+\|u_\e(t)\|_{H^1_0(\Omega)}^2\right)\, dt,\nonumber
\end{align}
which together with Gronwall's inequality yields (i) and (ii). 
Moreover, (iii) also follows from (i), (ii) and \eqref{bdd1}.
We next prove (iv). For any $\phi\in H^1_0(\Omega)$, the weak form \eqref{weakform} yields
\begin{align}
| \langle \partial_{tt}^2u_{\e}(t)+g(\mt)\partial_tu_\e(t), \phi\rangle_{H^1_0(\Omega)} |
\le	\|\phi\|_{H^1_0(\Omega)}\left(\|f_{\e}(t)\|_{H^{-1}(\Omega)}+\|u_{\e}(t)\|_{H^1_0(\Omega)}\right).\label{lem3.1(iv)}
\end{align}
Here we used the fact that
\begin{equation*}\label{ray}
|a(t,y)\xi\cdot \zeta|\le |\xi| |\zeta|
\quad \text{ for all $\xi,\zeta\in \R^N$ and a.e.~$(t,y)\in I\times\R^N$},
\end{equation*}
which follows from the Rayleigh-Ritz variational principle. By the boundedness of $(f_{\e})$ in $L^2(I;H^{-1}(\Omega))$ together with (i) and \eqref{lem3.1(iv)}, we deduce that 
\begin{align*}
\lefteqn{\int_0^T\|\partial_{tt}^2u_{\e}(t)+g(\mt)\partial_tu_\e(t)\|^2_{H^{-1}(\Omega)}\, dt}\\
&\stackrel{\eqref{lem3.1(iv)}}{\le}
\int_0^T\left(\|f_{\e}(t)\|_{H^{-1}(\Omega)}+\|u_{\e}(t)\|_{H^1_0(\Omega)}\right)^2\, dt
\le
2\left(\|f_{\e}\|_{L^2(I;H^{-1}(\Omega))}^2+\|u_{\e}\|_{L^2(I;H^1_0(\Omega))}^2\right),
\end{align*}
which implies that (iv) holds true. Here noting that 
\begin{align*}
\lefteqn{
\|\partial_{tt}^2u_{\e}+C_{\ast}t\e^{-r}\partial_{t}u_{\e}\|_{L^2(I;H^{-1}(\Omega))}}\\
&\quad \le
\|\partial_{tt}^2u_{\e}+g(\mt)\partial_{t}u_{\e}\|_{L^2(I;H^{-1}(\Omega))}
+\beta\|\partial_tu_\e\|_{L^2(I;H^{-1}(\Omega))}<C,
\end{align*}
we get (v) if $C_{\ast}=0$. As for $C_\ast\neq 0$, we infer that 
\begin{align}
\lefteqn{\int_0^Tt\left(\|\partial_{tt}^2u_\e(t)\|_{H^{-1}(\Omega)}^2+\| C_\ast \mt\partial_t u_\e(t)\|_{H^{-1}(\Omega)}^2\right)\, dt}\label{t-indepbdd}\\
&=
\int_0^Tt\left(\|\partial_{tt}^2u_\e(t)+ C_\ast \mt\partial_t u_\e(t)\|_{H^{-1}(\Omega)}^2
-2\bigl(\partial_{tt}^2u_\e(t),  C_\ast \mt\partial_t u_\e(t)\bigl)_{H^{-1}(\Omega)}\right)\, dt\nonumber\\
&\le 
CT-\int_0^T C_\ast\frac{t^2}{\e^r}\frac{d}{dt}\|\partial_tu_\e(t)\|_{H^{-1}(\Omega)}^2\, dt\nonumber\\
&=
CT-C_\ast T^2\|\sqrt{\e^{-r}}\partial_tu_\e(T)\|_{H^{-1}(\Omega)}^2
+2C_\ast
\int_0^T \|\sqrt{t\e^{-r}}\partial_tu_\e(t)\|_{H^{-1}(\Omega)}^2\, dt\le C.\nonumber
\end{align}
Hence we have
\begin{align}
\lefteqn{
\int_\sigma^T\left(\|\partial_{tt}^2u_\e(t)\|_{H^{-1}(\Omega)}^2+\|\mt C_\ast\partial_t u_\e(t)\|_{H^{-1}(\Omega)}^2\right)\, dt}
\label{t-indepbdd2}\\
&\le \frac{1}{\sigma}\int_0^Tt\left(\|\partial_{tt}^2u_\e(t)\|_{H^{-1}(\Omega)}^2+\|\mt C_\ast\partial_t u_\e(t)\|_{H^{-1}(\Omega)}^2\right)\, dt\le \frac{C}{\sigma}\nonumber
\end{align}
for all $0<\sigma<T$, which implies (v) and (vi).

\end{proof}

\begin{rmk}\label{test reg}
\rm
The argument mentioned above is not fully rigorous in view of the regularity of weak solutions. 
However, an approximate solution constructed by the Galerkin method in Theorem \ref{well-posedness} satisfies all the estimates and (weak) lower semicontinuity of norms assures the assertions.
\end{rmk}

Moreover, for smooth data, we have
\begin{lem}[Uniform estimates with smooth data]\label{bdd2}
Suppose that 
$a=a(y)$, $v_\e^0$, $v_\e^1$, $a(y)$, $g(s)$ and $f_\e$ are smooth, $(-\dv (a(\mx)\nabla v_\e^0))$, $(v_\e^1)$, 
$(f_\e)$ and $(\partial_t f_\e)$ are bounded in  $L^2(\Omega)$, $H^1_0(\Omega)$, $L^{\infty}(I;L^2(\Omega))$ and $L^2(\Omega\times I)$, respectively. Let $u_{\e}\in L^\infty(I;H^1_0(\Omega))$ be the unique weak solution to \eqref{DW}. Assume that {\bf (A)} holds. Let $I_{\sigma}=(\sigma,T)$ for any $\sigma>0$. Then the following {\rm(i)-(v)} hold\/{\rm :}
\begin{enumerate}
\item[(i)]
$(-\dv (a(\mx)\nabla u_\e))$ is bounded in $L^{\infty}(I;L^2(\Omega))$,
\item[(ii)]
$(\partial_tu_{\e})$ is bounded in $L^{\infty}(I;H^1_0(\Omega))$,
\item[(iii)]
$(\partial_{tt}^2u_{\e}+g(\mt)\partial_{t}u_{\e})$ is bounded in $L^{\infty}(I;L^2(\Omega))$,
\item[(iv)]
$(\partial_{tt}^2u_{\e})$ is bounded in $L^2(\Omega\times I_\sigma)$,
\item[(v)]
$(t\e^{-r}\partial_{t}u_{\e})$ is bounded in $L^2(\Omega\times I_\sigma)$, provided that $C_{\ast}\neq 0$.
\end{enumerate}
\end{lem}

\begin{proof}
Test \eqref{weakform} by $-\dv (a(\mx)\nabla \partial_t u_\e)$. Then we observe that
\begin{align*}
\lefteqn{\int_0^s\int_{\Omega}\partial_{tt}^2u_\e(x,t)\Bigl(-\dv \bigl(a(\mx)\nabla \partial_t u_\e(x,t)\bigl)\Bigl)\, dxdt}\\
&=
\int_0^s\int_{\Omega}\partial_t\bigl(\nabla \partial_t u_\e(x,t)\bigl)\cdot a(\mx)\nabla \partial_tu_\e(x,t)\, dxdt\\
&=
\frac{1}{2}\int_0^s\frac{d}{dt}\left(\int_{\Omega}a(\mx)\nabla\partial_tu_\e(x,t)\cdot\nabla \partial_tu_\e(x,t)\, dx\right)dt
\stackrel{\eqref{ellip}}{\ge}
\frac{\lambda}{2}\|\partial_tu_\e(s)\|_{H^1_0(\Omega)}^2-\frac{1}{2}\|v^1_\e\|_{H^1_0(\Omega)}^2
\end{align*}
and 
\begin{align*}
\lefteqn{\int_0^s\int_\Omega\left(-\dv \bigl(a(\mx)\nabla u_\e(x,t)\bigl)\right)\left(-\dv \bigl(a(\mx)\nabla \partial_tu_\e(x,t)\bigl)\right)\, dxdt}\\
&\quad=
\frac{1}{2}\|-\dv (a(\mx)\nabla u_\e(s))\|_{L^2(\Omega)}^2-\frac{1}{2}\|-\dv (a(\mx)\nabla v_\e^0)\|_{L^2(\Omega)}^2
\end{align*}
for all $s\in I$. Hence we derive that
\begin{align*}
\lefteqn{\frac{\lambda}{2}\|\partial_tu_\e(s)\|_{H^1_0(\Omega)}^2+\frac{1}{2}\|-\dv (a(\mx)\nabla u_\e(s))\|_{L^2(\Omega)}^2}\\
&\le
\frac{1}{2}\|v^1_\e\|_{H^1_0(\Omega)}^2+\frac{1}{2}\|-\dv (a(\mx)\nabla v^0_\e)\|_{L^2(\Omega)}^2\\
&\quad+
\int_{\Omega}f_\e(x,s)\Bigl(-\dv (a(\mx)\nabla u_\e(x,s))\Bigl)\, dx-\int_{\Omega} f_\e(x,0)\Bigl(-\dv (a(\mx)\nabla v_\e^0 (x) )\Bigl)\, dx\\
&\quad-
\int_0^s\int_{\Omega}\partial_t f_\e(x,t)\Bigl(-\dv (a(\mx)\nabla u_\e(x,t))\Bigl)\, dxdt\\
&\quad-
\int_0^s\int_{\Omega}g(\mt)\nabla \partial_t u_\e(x,t)\cdot a(\mx)\nabla \partial_tu_\e(x,t)\, dxdt\\
&\le
\frac{1}{2}\|v^1_\e\|_{H^1_0(\Omega)}^2+\|-\dv (a(\mx)\nabla v^0_\e)\|_{L^2(\Omega)}^2
\\
&\quad+
\|f_\e(s)\|_{L^2(\Omega)}^2+\frac{1}{4}\|-\dv (a(\mx)\nabla u_\e(s))\|_{L^2(\Omega)}^2
+\frac{1}{2}\|f_\e(0)\|_{L^2(\Omega)}^2\\
&\quad 
+\frac{1}{2}\|\partial_tf_\e\|_{L^2(\Omega\times I)}^2+\frac{1}{2}\int_0^s\|-\dv (a(\mx)\nabla u_\e(t))\|_{L^2(\Omega)}^2\, dt-\underbrace{\int_0^s\lambda g(\mt)\|\nabla\partial_t u_\e(t)\|_{L^2(\Omega)}^2\, dt}_{\ge 0}\\
&\le 
C+
\frac{1}{4}\|-\dv (a(\mx)\nabla u_\e(s))\|_{L^2(\Omega)}^2
+\frac{1}{2}\int_0^s\|-\dv (a(\mx)\nabla u_\e(t))\|_{L^2(\Omega)}^2\, dt,
\end{align*}
which together with Gronwall's inequality yields (i) and (ii). Hence one can derive that
\begin{equation}
\label{u1zero2}
\|\partial_{tt}^2u_\e+g(\mt)\partial_tu_{\e}\|_{L^{\infty}(I;L^2(\Omega))}\le
\|f_\e\|_{L^{\infty}(I;L^2(\Omega))}+\|-\dv(a(\mx)\nabla u_\e)\|_{L^{\infty}(I;L^2(\Omega))},
\end{equation}
which implies (iii). As in the proofs of \eqref{t-indepbdd} and \eqref{t-indepbdd2}, (iv) and (v) follow from \eqref{u1zero2}.

\end{proof}
Applying Lemma \ref{bdd}, we next get the following 
\begin{lem}[Weak(-star) and strong convergences]\label{conv}
Let $u_{\e}\in L^{\infty}(I;H^1_0(\Omega))$ be the unique weak solution of \eqref{DW} under the same assumption as in Lemma \ref{bdd}. 
Then there exist a subsequence $(\e_n)$ of $(\e)$, $u_{0}\in L^{\infty}(I;H^1_0(\Omega))$, 
 $w\in L^2(I;H^{-1}(\Omega))$ and $h \in L^2_{\rm loc}((0,T];H^{-1}(\Omega))$ such that, for any $\sigma\in I$, 
\begin{align}
\label{conv1}
u_{\e_n}&\to u_{0}\quad &&\text{ weakly-}\ast \text{ in }\ L^{\infty}(I;H^1_0(\Omega)),\\
\label{conv2}
\partial_{t}u_{\e_n}&\to \partial_{t}u_{0}\quad &&\text{ weakly-}\ast \text{ in }\ L^{\infty}(I;L^2(\Omega)),\\
\label{conv9}
\partial_{tt}^2u_{\e_n}+g(\mnt)\partial_{t}u_{\e_n}&\to w \quad &&\text{ weakly in }\ L^2(I;H^{-1}(\Omega)), \\
\label{conv3}
\partial_{tt}^2u_{\e_n}&\to \partial_{tt}^2u_{0} \quad &&\text{ weakly in }\ 
\begin{cases}
L^2(I;H^{-1}(\Omega)) &\text{ if }\ C_{\ast}= 0,\\
L^2(I_\sigma;H^{-1}(\Omega)) &\text{ if }\ C_{\ast}\neq 0,
\end{cases}  \\
\label{conv8}	
t\e_n^{-r}\partial_{t}u_{\e_n}&\to h \quad &&\text{ weakly in }\ L^2(I_\sigma;H^{-1}(\Omega)) \hspace{8mm}\text{ if }\ C_{\ast}\neq 0, \\
\label{conv4}
u_{\e_n}&\to u_{0}\quad &&\text{ strongly in }\ C(\overline{I};L^2(\Omega)),\\
\label{conv5}
\partial_{t}u_{\e_n}&\to \partial_tu_{0}\quad &&\text{ strongly in }	
\begin{cases}
C(\overline{I};H^{-1}(\Omega)) &\text{ if }\ C_{\ast}= 0,\\
C(\overline{I}_\sigma;H^{-1}(\Omega)) &\text{ if }\ C_{\ast}\neq 0,
\end{cases}\\
\label{conv7}
\sqrt{t}\partial_tu_{\e_n}&\to 0\quad &&\text{ strongly in }\ L^2(\Omega\times I)\hspace{15mm} \text{ if }\ C_{\ast}\neq 0.
\end{align}
In particular, if $C_{\ast} \neq 0$, then $\partial_t u_0(\cdot,t)\equiv 0$ for a.e.~$t\in I$, and hence, $u_0$ is independent of $t\in I$, i.e.,~$u_0=u_0(x)$. Furthermore, there exists $w_1\in L^{2}(\Omega\times I; H^{1}_{\rm per}(\square\times J)/\R)$ such that 
\begin{align}
	 \partial_t u_{\e_n}  &\wtts \partial_t u_{0}+\partial_s  w_1  \quad &&\text{ in }\ L^2(\Omega\times I\times \square\times J), \label{conv6.5}\\
	a(t,\mnx)\nabla u_{\e_n}&\wtts a(t,y)(\nabla u_{0}+\nabla_y  w_1  ) \quad &&\text{ in }\ [L^2(\Omega\times I\times \square\times J)]^N.\label{conv5.5}
\end{align}
Thus it holds that
\begin{align}
a(t,\mnx)\nabla u_{\e_n}&\to \langle a(t,\cdot)(\nabla u_{0}+\nabla_y  w_1  )\rangle_{y,s} \quad \text{ weakly in }\ [L^2(\Omega\times I)]^N,\label{conv6}
\end{align}
where
$$
\bigl\langle a(t,\cdot)\bigl(\nabla u_0(x,t)+\nabla_y  w_1(x,t,\cdot,\cdot)  \bigl)\bigl\rangle_{y,s}
=
\int_{\square}a(t,y)\bigl(\nabla u_0(x,t)+\nabla_y  \langle w_1  (x,t,y,\cdot)\rangle_s\bigl)\, dy.
$$
\end{lem}

\begin{proof}
Thanks to Lemma \ref{bdd}, we readily obtain \eqref{conv1}-\eqref{conv8}. Furthermore, from (i) and (ii) of Lemma \ref{bdd}, the Asocili-Arzel\'a theorem yields \eqref{conv4}. In the same way, \eqref{conv5} also holds true by (ii) and (v) of Lemma \ref{bdd}. As for \eqref{conv7}, noting by (iii) of Lemma \ref{bdd} that
$$
\limsup_{\e_n\to 0_+}\|\sqrt{t}\partial_t u_{\e_n}\|_{L^2(\Omega\times I)}^2\le 
\limsup_{\e_n\to 0_+} C\e^{r}_n=0,
$$
we obtain \eqref{conv7}. Thus $u_0=u_0(x)$, provided that $C_{\ast}\neq 0$. We finally show  \eqref{conv6.5}, \eqref{conv5.5} and \eqref{conv6}. All the assumptions of Theorem \ref{gradientcpt} can be checked by (i) and (ii) of Lemma \ref{bdd}. 
Hence \eqref{conv6.5} holds true. Moreover, note that, for any $\Psi\in [L^2_{\rm per}(\square\times J;C_{\rm c}(\Omega\times I))]^N$, $\Psi$ and $a(t,y)\Psi$ are admissible test functions in $[L^2(\Omega\times I\times \square\times J)]^N$ (see \cite[Theorems 2 and 4]{LNW} for details) and define $\Xi\in [L^2(\Omega\times I\times \square\times J)]^N$ by
$$
a(t,\mnx)\nabla u_{\e_n} \wtts \Xi \quad \text{ in }\ [L^2(\Omega\times I\times \square\times J)]^N.
$$
Then, Theorem \ref{gradientcpt} yields that 
\begin{align*}
\lefteqn{\int_0^T\int_\Omega\int_0^1\int_\square\Xi(x,t,y,s)\cdot \Psi(x,t,y,s)\, dZ}\\
&\quad=
\lim_{\e_n\to 0_+}\int_0^T\int_{\Omega}\nabla u_{\e_n}(x,t)\cdot \tenchi a(t,\mnx)\Psi(x,t,\mnx,\mnt)\, dxdt\\
&\quad=
\int_0^T\int_{\Omega}\int_0^1\int_\square\bigl(\nabla u_{0}(x,t)+\nabla_y w_1(x,t,y,s)\bigl)\cdot \tenchi a(t,y)\Psi(x,t,y,s)\, dZ,
\end{align*}
which implies \eqref{conv5.5}, and hence, (i) of Remark \ref{indepwtts} yields \eqref{conv6}. This completes the proof.
\end{proof}

\section{Proof of Theorem \ref{HPthm}}
We first  derive   the homogenized equation by setting
\begin{align*}
j_{\rm hom}(x,t):=\Bigl\langle a(t,\cdot)\bigl(\nabla u_0(x,t)+\nabla_y  w_1 (x,t,\cdot,\cdot)\bigl)\Bigl\rangle_{y,s}.
\end{align*}
Recalling \eqref{conv9} and \eqref{conv6}, we observe that, for all $\phi\in H^1_0(\Omega)$ and $\psi\in C^{\infty}_{\rm c}(I)$,
\begin{align*}
\lefteqn{\int_0^T \langle f(t), \phi\rangle_{H^1_0(\Omega)}\psi(t)\, dt}\\
&=
\lim_{\e_n\to 0_+}\int_0^T \langle f_{\e_n}(t), \phi\rangle_{H^1_0(\Omega)}\psi(t)\, dt\\
&\stackrel{\eqref{weakform}}{=}
\lim_{\e_n\to 0_+}\int_0^T
\Bigl[
\langle \partial_{tt}^2u_{\e_n}(t)+g(\mnt)\partial_tu_{\e_n}(t),\phi\rangle_{H^1_0(\Omega)}
+\bigl(a(t,\mnx)\nabla u_{\e_n}(t),  \nabla \phi\bigl)_{L^2(\Omega)}\Bigl]\psi(t)\, dt\\
&\stackrel{\eqref{conv9}, \eqref{conv6}}{=}
\int_0^T
\Bigl[
\langle  w , \phi\rangle_{H^1_0(\Omega)}
+\bigl(j_{\rm hom}(t),  \nabla \phi\bigl)_{L^2(\Omega)}\Bigl]\psi(t)\, dt.\nonumber
\end{align*}
Here $w$ can be regarded as  
\begin{equation}
w=\partial_{tt}^2u_{0}+\langle g_{\rm per}\rangle_s\partial_t 
u_0+C_{\ast} h.\label{conv-w}
\end{equation}
Actually, due to $\psi\in C^{\infty}_{\rm c}(\Omega)$, this follows from \eqref{conv8}, \eqref{conv5} and Proposition \ref{mean}. Hence, by the arbitrariness of $\psi\in C^{\infty}_{c}(I)$, $u_0$ turns out to be a weak solution to
\begin{equation}\label{DWHP}
\left\{
\begin{aligned}
&\partial_{tt}^2u_{0}-\dv\, j_{\rm hom}+\langle g_{\rm per}\rangle_s\partial_tu_0+C_{\ast} h=f \quad\text{ in } \Omega\times I, \\
&u_{0}|_{\partial\Omega}=0 , \quad 	u_{0}|_{t=0}=v^0,\quad \partial_tu_{0}|_{t=0}= \tilde{v}^1, 
\end{aligned}
\right.
\end{equation}
where
$$
\tilde{v}^1=
\begin{cases}
v^1 &\text{ if }\ C_{\ast}=0,\\
0 &\text{ if }\ C_{\ast}\neq 0.
\end{cases}
$$
Indeed, noting that
\begin{align*}
\|u_0(0)-v^0\|_{L^2(\Omega)}
&\le
\|u_0(0)-u_{\e_n}(0)\|_{L^2(\Omega)}
+\|u_{\e_n}(0)-v^0\|_{L^2(\Omega)}\\
&\le
\|u_0-u_{\e_n}\|_{C(\overline{I};L^2(\Omega))}
+\|v_{\e_n}^0-v^0\|_{L^2(\Omega)},
\end{align*}
we see by \eqref{conv4} and {\bf (A)} that
\begin{align*}
\|u_0(0)-v^0\|_{L^2(\Omega)}
\le
\limsup_{\e_n\to 0_+}\|u_0-u_{\e_n}\|_{C(\overline{I};L^2(\Omega))}
+\limsup_{\e_n\to 0_+}\|v_{\e_n}^0-v^0\|_{L^2(\Omega)}=0,
\end{align*}
which implies that $u_0(x,0)=v^0$. Thus $u_0\equiv v^0$ by $\partial_t u_0\equiv 0$, provided that for $C_{\ast}\neq0$. To check $\partial_t u_0(x,0)=v^1(x)$ a.e.~in $\Omega$ for $C_{\ast}=0$, let $\psi\in C^{\infty}(I)$ be such that $\psi(T)=0$ and $\psi(0)=1$. Then we infer that, for all $\phi\in C^{\infty}_{\rm c}(\Omega)$,
\begin{align*}
\int_{\Omega} v^1(x)\phi(x)\, dx
&\stackrel{{\bf (A)},\, \eqref{weakform}}{=}
\lim_{\e_n\to 0_+}\int_{\Omega} v_{\e_n}^1(x)\phi(x)\, dx\\
&\quad+
\lim_{\e_n\to 0_+}\int_0^T\Bigl\langle
\partial_{tt}^2 u_{\e_n}(t)
+g(\mnt)\partial_t u_{\e_n}(t),\phi\Bigl\rangle_{H^1_0(\Omega)}\psi(t)\, dt\\
&\quad+
\lim_{\e_n\to 0_+}\int_0^T\int_{\Omega}\Bigl[ a(t,\mnx)\nabla u_{\e_n}(x,t)\cdot\nabla \phi(x)\psi(t)-f_{\e_n}(x,t)\phi(x)\psi(t)\Bigl]\, dxdt\\
&=
\lim_{\e_n\to 0_+}\int_0^T\int_{\Omega} \Bigl[-\partial_t u_{\e_n}(x,t)\phi(x)\partial_t\psi(t)+g(\mnt)\partial_t u_{\e_n}(x,t)\phi(x)\psi(t)\\
&\quad+
a(t,\mnx)\nabla u_{\e_n}(x,t)\cdot\nabla \phi(x)\psi(t)
-f_{\e_n}(x,t)\phi(x)\psi(t)\Bigl]\, dxdt\\
&=
\int_0^T\int_{\Omega} \Bigl[-\partial_t u_{0}(x,t)\phi(x)\partial_t\psi(t)+\langle g_{\rm per}\rangle_s\partial_tu_0(x,t)\phi(x)\psi(t)\\
&\quad+
j_{\rm hom}(x,t)\cdot\nabla \phi(x)\psi(t)-f(x,t)\phi(x)\psi(t)\Bigl]\, dxdt\\
&\stackrel{\eqref{DWHP}}{=}
\int_{\Omega} \partial_t u_0(x,0)\phi(x)\, dx,  
\end{align*}
which together with the arbitrariness of $\phi\in C^{\infty}_{\rm c}(\Omega)$ yields that $\partial_tu_0(x,0)=v^1(x)$ a.e.~in $\Omega$ for $C_{\ast}= 0$. 

The rest of the proof is to show that
\begin{equation}\label{rest}
j_{\rm hom}=a_{\rm hom}(t)\nabla u_0(x,t).
\end{equation}
Here $a_{\rm hom}(t)$ is the homogenized matrix defined by \eqref{a_hom}. Thus it suffices to prove \eqref{HPu1}, that is,
\begin{equation}
\langle w_1\rangle_{s}=u_1:=\sum_{k=1}^N\partial_{x_k}u_0(x,t)\Phi_k(t,y),\label{u1}
\end{equation}
where $\Phi_k$ is the corrector defined by either \eqref{CPslow} or \eqref{CPcritical}. 
Indeed, if \eqref{u1} holds, then we derive that
\begin{eqnarray*}
j_{\rm hom}(x,t)
&=&
\left\langle a(t,\cdot)\bigl(\nabla u_0(x,t)+\nabla_y w_1(x,t,\cdot,\cdot) \bigl)\right\rangle_{y,s}\\
&\stackrel{\eqref{u1}}{=}&
\int_{\square}a(t,y)\Bigl(\nabla u_0(x,t)+\sum_{k=1}^N\partial_{x_k}u_0(x,t)\nabla_y\Phi_k(t,y)\Bigl)\, dy\\
&=&
\sum_{k=1}^N\underbrace{\Bigl(\int_{\square}a(t,y)\left(\nabla_y\Phi_k(t,y)+e_k\right)\, dy\Bigl)}_{=a_{\rm hom}(t)e_k \text{ by \eqref{a_hom}}}\partial_{x_k}u_0(x,t)=a_{\rm hom}(t)\nabla u_0(x,t),					
\end{eqnarray*}
which implies \eqref{rest}. Hence $u_0$ turns out to be a unique weak solution to \eqref{HDW2}. 
Indeed, this follows from the uniqueness of the corrector $\Phi_k$ and the similar argument as in Theorem \ref{well-posedness} if $C_{\ast}=0$ and $u_0\equiv v^0$ whenever $C_{\ast}\neq 0$. Thus we have
$$
u_{\e} \to u_0\quad \text{ as }\ \e\to 0_+
$$
without taking any subsequence $(\e_n)$. 
Therefore, \eqref{HPconv1}--\eqref{HPconv3} hold by Lemma \ref{conv} and \eqref{conv-w}. Thus we get all the assertions. 

In the rest of this section, we shall prove  \eqref{u1}  for all $0<r<+\infty$. To this end, we show the following 
\begin{lem}\label{keylem}
Under the same assumption as in Theorem \ref{HPthm}, it holds that
\begin{align}
\lefteqn{	
\lim_{\e_n\to 0_+}\e_n^{1-r}\int_0^T\int_{\Omega}\Bigl[-\partial_t u_{\e_n}(x,t)	\partial_s c(\mnt)
+C_{\ast}t\partial_tu_{\e_n}(x,t)c(\mnt)\Bigl]\phi(x)b(\mnx)\psi(t)\, dxdt
}\label{key}\\
&\quad+
\lim_{\e_n\to 0_+}\int_0^T\int_{\Omega}	a(t,\mnx)\nabla u_{\e_n}(x,t)  \cdot   \phi(x)\nabla_y b(\mnx)\psi(t)c(\mnt)\, dxdt
=0\nonumber
\end{align}
for all $\phi\in C^{\infty}_c(\Omega)$, $b\in C^{\infty}_{\rm per}(\square)$, $\psi \in  C_{\rm c}^{\infty}(I)  $ and $c\in C^{\infty}_{\rm per}(J)$.
\end{lem}

\begin{proof}
Taking a difference of the weak forms for \eqref{DW} and \eqref{DWHP} and recalling $w$ in \eqref{conv-w}, we observe that
\begin{align*}
0
&=
\int_0^T\Bigl\langle
\partial_{tt}^2u_{\e_n}(t)+g(\mnt)\partial_t u_{\e_n}(t)-w(t), \phi b(\tfrac{\cdot}{\e_n})\Bigl\rangle_{H^1_0(\Omega)}\psi(t)c(\mnt)\, dt\displaybreak[0]\\
&\quad +
\int_0^T\int_{\Omega}	\bigl(  a(t,\tfrac{x}{\e_n})\nabla u_{\e_n}(x,t)-j_{\rm hom}(x,t)  \bigl) \cdot  \nabla\bigl(\phi(x)b(\mnx)\bigl)\psi(t)c(\mnt)	\, dxdt\displaybreak[0]\\
&\quad -
\int_0^T\int_{\Omega}	(f_{\e_n}-f)(x,t) \phi(x)b(\mnx)\psi(t)c(\mnt)\, dxdt\displaybreak[0]\\
&=-\int_0^T\int_{\Omega}	\partial_t u_{\e_n}(x,t)	\phi(x)b(\mnx)\bigl(\partial_{t}\psi(t)c(\mnt)+\psi(t)\e_n^{-r}\partial_s c(\mnt)\bigl)\, dxdt\displaybreak[0]\\
&\quad +
\int_0^T\int_{\Omega}	\bigl(g_{\rm per}(\mnt)+C_{\ast}\mnt\bigl)\partial_t u_{\e_n}(x,t) \phi(x)b(\mnx)\psi(t)c(\mnt)\, dxdt\displaybreak[0]\\
&\quad -
\int_0^T\Bigl\langle w(t), \phi b(\tfrac{\cdot}{\e_n})\Bigl\rangle_{H^1_0(\Omega)}\psi(t)c(\mnt)\, dt\displaybreak[0]\\
&\quad +
\int_0^T\int_{\Omega}	\bigl(a(t,\mnx)\nabla u_{\e_n}(x,t)-j_{\rm hom}(x,t)\bigl)  \cdot   \bigl(\nabla\phi(x)b(\mnx)+\e_n^{-1}\phi(x)\nabla_y b(\mnx)\bigl) \psi(t)c(\mnt)\, dxdt\displaybreak[0]\\
&\quad -
\int_0^T\int_{\Omega}	(f_{\e_n}-f)(x,t) \phi(x)b(\mnx)\psi(t)c(\mnt)\, dxdt.
\end{align*}
Multiplying both sides by $\e_n$, we conclude that
\begin{align}
\lefteqn{	
-\e_n^{1-r}\int_0^T\int_{\Omega}	\partial_t u_{\e_n}(x,t)	\phi(x)b(\mnx)\psi(t)\partial_s c(\mnt)\, dxdt
}\displaybreak[0]\label{key3}\\
&\quad+
\int_0^T\int_{\Omega}	a(t,\mnx)\nabla u_{\e_n}(x,t)  \cdot   \phi(x)\nabla_y b(\mnx)\psi(t)c(\mnt)\, dxdt\displaybreak[0]\nonumber\\
&\quad+
\e_n^{1-r}\int_0^T\int_{\Omega} C_{\ast}t\partial_tu_{\e_n}(x,t)   \phi(x)b(\mnx)\psi(t)c(\mnt)\, dxdt\displaybreak[0]\nonumber\\
&=
\e_n\int_0^T\int_{\Omega}	\partial_t u_{\e_n}(x,t)	\phi(x)b(\mnx)\partial_{t}\psi(t)c(\mnt)\, dxdt\displaybreak[0]\nonumber\\
&\quad -
\e_n\int_0^T\int_{\Omega}	g_{\rm per}(\mnt)\partial_t u_{\e_n}(x,t) \phi(x)b(\mnx)\psi(t)c(\mnt)\, dxdt\displaybreak[0]\nonumber\\
&\quad+\e_n
\int_0^T\Bigl\langle w(t), \phi b(\tfrac{\cdot}{\e_n})\bigl\rangle_{H^1_0(\Omega)}\psi(t)c(\mnt)\, dt
\displaybreak[0]\nonumber\\
&\quad-
\e_n\int_0^T\int_{\Omega}	\bigl(a(t,\mnx)\nabla u_{\e_n}(x,t)-j_{\rm hom}(x,t)\bigl)  \cdot   \nabla\phi(x)b(\mnx) \psi(t)c(\mnt)\, dxdt\displaybreak[0]\nonumber\\
&\quad+
\int_0^T\int_{\Omega}	j_{\rm hom}(x,t)  \cdot   \phi(x)\nabla_y b(\mnx)\psi(t)c(\mnt)\, dxdt\displaybreak[0]\nonumber\\
&\quad +\e_n
\int_0^T\int_{\Omega}	(f_{\e_n}-f)(x,t) \phi(x)b(\mnx)\psi(t)c(\mnt)\, dxdt\to 0\quad \text{ as }\ \e_n\to 0_+.\displaybreak[0]\nonumber
\end{align}
Here we used {\bf (A)}, Lemmas \ref{bdd} and \ref{conv}, Proposition \ref{mean}  and  $\langle\nabla_y b\rangle_y=0$.
\end{proof}

Employing Lemma \ref{conv} and Corollary \ref{veryweak}, we shall apply \eqref{key} for any $\phi\in C^{\infty}_c(\Omega)$, $b\in C^{\infty}_{\rm per}(\square)/\R$, $\psi \in C^{\infty}_{\rm c}(I)$ and $c\in C^{\infty}_{\rm per}(J)$ to show \eqref{u1}.
\begin{lem}\label{1st}
For any $0<r\le 1$, \eqref{u1} holds.
\end{lem}
\begin{proof}
Set $c(s)\equiv 1$ in \eqref{key}. By \eqref{conv7}, the first term in \eqref{key} is zero. Thanks to \eqref{conv5.5}, we deduce by \eqref{key} that 
\begin{equation*}
\int_0^T\int_{\Omega}\int_0^1\int_{\square}	a(t,y)\bigl(\nabla u_0(x,t)+\nabla_y  w_1  (x,t,y,s)\bigl)\cdot \phi(x)\nabla_y b(y)\psi(t)c(s)\, dZ=0.
\end{equation*}
From the arbitrariness of $\phi\in C^{\infty}_{\rm c}(\Omega)$ and $\psi\in C^{\infty}_{\rm c}(I)$, we get
\begin{equation}\label{u1eq}
\int_{\square}	a(t,y)\bigl(\nabla u_0(x,t)+\nabla_y  \langle w_1(x,t,y,\cdot)\rangle_s  \bigl)\cdot \nabla_y b(y)\, dy=0\quad \text{ a.e. in }\ \Omega\times I.
\end{equation}
Recalling that   
\begin{equation*}
u_1=\sum_{k=1}^N\partial_{x_k}u_0(x,t)\Phi_k(y),
\end{equation*}
where $\Phi_k$ is the solution to \eqref{CPslow},
we check that
\begin{eqnarray*}
\lefteqn{\int_{\square}	a(t,y)\bigl(\nabla u_0(x,t)+\nabla_yu_1(x,t,y)\bigl)\cdot \nabla_yb(y)\, dy}\\
&\quad=&
\sum_{k=1}^N\partial_{x_k}u_0(x,t)\int_{\square}	a(t,y)\bigl(\nabla_y\Phi_k(y)+e_k\bigl)\cdot \nabla_yb(y)\, dy	
\stackrel{\eqref{CPslow}}{=}0.
\end{eqnarray*}
Hence  \eqref{u1eq} with $\langle w_1\rangle_s$ replaced by $u_1(x,t,y)$ holds. Setting $b=( \langle w_1\rangle_s  -u_1)(x,t,\cdot)$ and subtracting \eqref{u1eq} for  $\langle w_1\rangle_y$ and $u_1$,  we deduce by the Poincar\'e-Wirtinger inequality that 
\begin{eqnarray}\label{ztilz}
0
&=&
\int_{\square}a( t,  y)\nabla_y(  \langle w_1\rangle_s -u_1 )(x,t,y)\cdot \nabla_y(  \langle w_1\rangle_s -u_1  )(x,t,y)\, dy\nonumber\\
&\stackrel{\eqref{ellip}}{\ge}& 
{\lambda}\|\nabla_y( \langle w_1\rangle_s -u_1  )(x,t)\|^2_{L^2(\square)}
\ge
\frac{\lambda}{C_{\square}}\|(  \langle w_1\rangle_s -u_1  )(x,t)\|^2_{L^2(\square)},
\nonumber
\end{eqnarray}
which implies that $\langle w_1\rangle_s =u_1$. This completes the proof.
\end{proof}

Before discussing the case $r>1$, we claim that
$$
w_1=w_1(x,t,y) \quad \text{ for all $r\in (1,+\infty)$}.
$$
Indeed, multiplying both sides by $\e_n^{-2(1-r)}$ in  \eqref{key3}, we see that the third term in \eqref{key3} is zero as $\e_n\to 0_+$ due to \eqref{conv7}, and then, one can derive by Lemma \ref{bdd} and Corollary \ref{veryweak} that
\begin{align*}
0&=
-\lim_{\e_n\to 0_+}
\e_n^{r-1}\int_0^T\int_{\Omega}	\partial_t u_{\e_n}(x,t)	\phi(x)b(\mnx)\psi(t)\partial_s c(\mnt)\, dxdt\\
&=
\underbrace{\lim_{\e_n\to 0_+}\e_n^{r-1}\int_0^T\int_{\Omega}	 u_{\e_n}(x,t)	\phi(x)b(\mnx)\partial_t\psi(t)\partial_s c(\mnt)\, dxdt}_{=0}\\
&\quad+\lim_{\e_n\to 0_+}\int_0^T\int_{\Omega}	 \frac{u_{\e_n}}{\e_n}(x,t)	\phi(x)b(\mnx)\psi(t)\partial_{ss}^2 c(\mnt)\, dxdt\\
&=
\int_0^T\int_{\Omega}\int_0^1\int_{\square}	  w_1(x,t,y,s) 	\phi(x)b(y)\psi(t)\partial_{ss}^2 c(s)\, dZ,
\end{align*} 
which implies that $\partial_s  w_1$ is independent of $s\in J$ and so is $  w_1  $ by $J$-periodicity. Thus $ w_1  \in L^2(\Omega\times I;H^1_{\rm per}(\square)/\R)$ for all $r>1$.

We choose $c(s)\equiv 1$ in \eqref{key} below. Then one can get the following 
\begin{lem}\label{3rd}
For any $1<r\le 2$, \eqref{u1} holds. 
\end{lem}
\begin{proof}
As for the periodic case, \eqref{u1eq} follows from \eqref{key} and \eqref{conv5.5} with $c(s)\equiv 1$ and $C_\ast = 0$. Thus the assertion is obtained as in the proof of Lemma \ref{1st}. We next consider the quasi-periodic case. Applying Corollary \ref{veryweak} to \eqref{key} with $c(s)\equiv 1$, we deduce that
\begin{align}
\lefteqn{\lim_{\e_n\to 0_+}
\e_n^{1-r}\int_0^T\int_{\Omega} C_{\ast}t\partial_tu_{\e_n}(x,t)\phi(x)b(\mnx)\psi(t)\, dxdt}\label{key2}\\
&=-\lim_{\e_n\to 0_+}
\e_n^{2-r}\int_0^T\int_{\Omega} C_{\ast}\frac{u_{\e_n}}{\e_n}(x,t)\phi(x)b(\mnx)\partial_t\bigl(t\psi(t)\bigl)\, dxdt\nonumber\\
&=
\begin{cases}
0\quad &\text{ if $1<r<2$,}\vspace{3mm}\\
\displaystyle -\int_0^T\int_{\Omega}\int_{\square}C_{\ast} w_1  (x,t,y)\phi(x)b(y)\partial_t\bigl(t\psi(t)\bigl)\, dydxdt
\quad &\text{ if $r=2$.}
\end{cases}
\nonumber
\end{align} 
Thus \eqref{conv5.5} and \eqref{key} yield \eqref{u1eq} for the case $1<r<2$. 

On the other hand, for $r=2$, we find by \eqref{conv5.5}, \eqref{key} and \eqref{key2} that
\begin{align*}
\lefteqn{\int_0^T\int_{\Omega}\bigl\langle C_{\ast}t\partial_t w_1(x,t,\cdot),b\bigl\rangle_{H^1_{\rm per}(\square)/\R}\phi(x)\psi(t)\, dxdt}\\
&\quad+
\int_0^T\int_{\Omega}\int_{\square}	 a(t,y)\bigl(\nabla u_0(x)+\nabla_y w_1(x,t,y)\bigl)\cdot \phi(x)\nabla_y b(y)\psi(t)\, dydxdt
=0.
\nonumber
\end{align*}
Here we used the fact that
\begin{align}
\lefteqn{
-C_{\ast}\int_0^T\int_{\Omega}\int_\square
w_1(x,t,y)\phi(x)b(y)\partial_t\bigl(t\psi(t)\bigl)\, dydxdt}\label{chk-t-d}\\
&\quad=
C_{\ast}\int_0^T\int_{\Omega}
\bigl\langle t\partial_tw_1(x,t,\cdot), b\bigl\rangle_{H^1_{\rm per}(\square)/\R}\phi(x)\psi(t)\, dxdt.\nonumber
\end{align}
Indeed, define $\xi(x,t,\cdot)\in (H^1_{\rm per}(\square)/\R)^{\ast}$ by
$$
\int_0^T\bigl\langle  \xi(x,t,\cdot), \zeta(t,\cdot)\bigl\rangle_{H^1_{\rm per}(\square)/\R}\, dt
=
\int_0^T\int_{\square}a(t,y)\bigl(\nabla u_0(x)+\nabla_y w_1  (x,t,y)\bigl)\cdot\nabla_y \zeta(t,y)\, dydt
$$
for $\zeta\in L^2(I;H^1_{\rm per}(\square)/\R)$. Here $(H^1_{\rm per}(\square)/\R)^{\ast}$ is the dual space of $H^1_{\rm per}(\square)/\R$. By Pettis's theorem, we see that $\xi:\Omega\to L^2(I;(H^1_{\rm per}(\square)/\R)^{\ast}) $ is measurable, and moreover, $\xi \in L^2(\Omega;L^2(I;(H^1_{\rm per}(\square)/\R)^{\ast}))$ due to $u_0 \in H^1_0(\Omega)$ and $\nabla_y  w_1  \in  [L^2(\Omega \times I \times \square)]^N$.   
Hence one can verify by  \eqref{key} and  \eqref{key2} that 
\begin{equation*}
C_\ast\int^T_0    w_1(x,t,\cdot) \partial_t \bigl(t\psi(t)\bigl) \, dt = \int^T_0 \xi(x,t,\cdot)  \psi(t)\, dt \ \mbox{ in } (H^1_{\rm per}(\square)/\R)^{\ast}. 
\end{equation*}
Since $\psi \in C^\infty_{\rm c}(I)$ is arbitrary, we have
\begin{equation*}
C_\ast t\partial_t  w_1  (x,t,\cdot)=-\xi(x,t,\cdot)\ \mbox{ in } (H^1_{\rm per}(\square)/\R))^{\ast}
\end{equation*}
in the distributional sense  for a.e.~$(x,t) \in \Omega \times I $. Thus \eqref{chk-t-d} follows, and then, the arbitrariness of $\phi\in C^{\infty}_{\rm c}(\Omega)$ yields that
\begin{align}\label{CP2}
\lefteqn{\int_0^T \bigl\langle C_{\ast}t\partial_t w_1(x,t,\cdot),b\bigl\rangle_{H^1_{\rm per}(\square)/\R} \psi(t)\, dt}\\
&\quad +
\int_0^T\int_{\square}	a(t,y)\bigl(\nabla u_0(x)+\nabla_y  w_1  (x,t,y)\bigl)\cdot \nabla_y b(y)\psi(t)\, dydt=0\quad \text{ a.e.~in $\Omega$.}\nonumber
\end{align}
 
Now, recall \eqref{HPu1}, where $\Phi_k$ is the solution to \eqref{CPcritical}. 
Then \eqref{CP2} with $ w_1$ replaced by $u_1(x,t,y)$ holds. Hence choosing $b\psi=( w_1- u_1 )(x,\cdot,\cdot)$ and subtracting \eqref{CP2} for   $w_1$ and $u_1$, we derive that 
\begin{align}
0&=
\int_0^T\frac{C_{\ast}t}{2}\frac{d}{dt}\|(  w_1- u_1  )(x,t)\|_{L^2(\square)}^2\, dt \label{unique2}\\
&\quad +
\int_0^T\int_{\square}a(t,y)\nabla_y(  w_1- u_1 )(x,t,y)\cdot \nabla_y(  w_1- u_1 )(x,t,y)\, dydt\nonumber\\
&\stackrel{\eqref{ellip}}{\ge}
\frac{C_{\ast}T}{2}\|( w_1- u_1  )(x,T)\|_{L^2(\square)}^2-\frac{C_{\ast}}{2}\int_0^T\|( w_1- u_1  )(x,t)\|_{L^2(\square)}^2\, dt\nonumber\\
&\quad +
{\lambda}\int_0^T\|\nabla_y( w_1-u_1  )(x,t)\|^2_{L^2(\square)}\, dt\nonumber\\
&\stackrel{}{\ge}
\frac{C_{\ast}T}{2}\|( w_1- u_1  )(x,T)\|_{L^2(\square)}^2+\underbrace{\left(-\frac{C_{\ast}}{2}+\frac{\lambda}{C_{\square}}\right)}_{\ge 0\ \text{ by \bf (A)}}\int_0^T\|( w_1-u_1  )(x,t)\|_{L^2(\square)}^2\, dt,\nonumber
\end{align}
which implies $ w_1=u_1$. Furthermore, \eqref{unique2} yields the uniqueness of \eqref{CPcritical}, and moreover, if $a=a(y)$, then $\Phi_k$ is the solution to \eqref{CPslow} due to the uniqueness of \eqref{CPcritical}. This completes the proof. 
\end{proof}

We finally discuss the case where $2<r<+\infty$. 
\begin{lem}\label{4th}
For any $2<r< +\infty$,  \eqref{u1}  holds. 
\end{lem}
\begin{proof}
Due to \eqref{conv5.5} and \eqref{key}, it suffices to check 
$$
\lim_{\e_n\to 0+}\e_n^{1-r}\int_0^T\int_{\Omega}C_{\ast}t\partial_tu_{\e_n}(x,t)\phi(x)b(\mnx)\psi(t)c(\mnt)\, dxdt
=0
$$
with $c(s)\equiv 1$. It is clear if $C_{\ast}=0$. If $C_{\ast}\neq 0$,
since $(t\e^{-r}\partial_{t}u_\e)$ is bounded in $L^2(\Omega\times I_\sigma)$ by (v) of Lemma \ref{bdd2}, we find by $\psi\in C^{\infty}_{\rm c}(I)$ that
\begin{equation*}
\e_n^{1-r}\left|\int_0^T\int_{\Omega}C_{\ast}t\partial_tu_{\e_n}(x,t)\phi(x)b(\mnx)\psi(t)\, dxdt\right|
\le
C\e_n\|t\e_n^{-r}\partial_tu_{\e_n}\psi\|_{L^1(\Omega\times I)}\to 0\, \text{ as $\e_n\to 0_+$},
\end{equation*}
which completes the proof. 
\end{proof}

By Lemmas \ref{1st}, \ref{3rd} and \ref{4th}, we obtain \eqref{u1} for all $r\in (0,+\infty)$. 
Therefore, Theorem \ref{HPthm} is proved.

\section{Proof of Proposition \ref{property_of_a_hom}}
We consider the case of $C_{\ast}\neq 0$ and $r=2$ only (see \cite[Proposition 1.8]{AO} for the proof of the other case). We first prove (i). For each $\xi\in\R^N$, there exists a unique solution $\Phi_{\xi}\in L^{2}(I;H^1_{\rm per}(\square)/\R)$ to 
\begin{align}\label{criticalxi}
C_{\ast}t\partial_t\Phi_{\xi}-\dv_y\bigl[a(t,y)(\nabla_y \Phi_{\xi}+\xi)\bigl]=0\quad \text{ in }  I\times \square. 
\end{align}
Using \eqref{a_hom} and \eqref{ellip}, we derive that, for a.e. $t\in I$,
\begin{eqnarray*}  
a_{\rm hom}(t)\xi\cdot\xi
&\stackrel{\eqref{a_hom}}{=}&
\int_{\square} a(t,y)\bigl(\nabla_y\Phi_{\xi}(t,y)+\xi\bigl)\cdot\ \xi\, dy 	\nonumber\\
&\stackrel{\eqref{criticalxi}}{=}&
\int_{\square}a(t,y)(\nabla_y\Phi_{\xi}(t,y)+\xi)\cdot(\nabla_y\Phi_{\xi}(t,y)+\xi)\, dy+\frac{C_{\ast}t}{2}\frac{d}{dt}\|\Phi_{\xi}(t)\|_{L^2(\square)}^2\nonumber\\
&\stackrel{\eqref{ellip}}{\ge}&
\lambda\int_{\square}|\xi+\nabla_y\Phi_{\xi}(t,y)|^2\, dy+\frac{C_{\ast}t}{2}\frac{d}{dt}\|\Phi_{\xi}(t)\|_{L^2(\square)}^2
\nonumber\\
&=& 
\lambda\left(|\xi|^2+ \|\nabla_y\Phi_{\xi}(t)\|^2_{L^2(\square)}\right)
+\frac{C_{\ast}t}{2}\frac{d}{dt}\|\Phi_{\xi}(t)\|_{L^2(\square)}^2.
\end{eqnarray*}
Here we used the fact that $\langle\nabla_y\Phi_{\xi} (t,\cdot)  \rangle_y=0$. In an analogous way, we get
$$
a_{\rm hom}(t)\xi\cdot\xi
\le
\left(|\xi|^2+ \|\nabla_y\Phi_{\xi}(t)\|^2_{L^2(\square)}\right)+\frac{C_{\ast}t}{2}\frac{d}{dt}\|\Phi_{\xi}(t)\|_{L^2(\square)}^2.
$$ 

We next prove (ii). Let $\Phi_{j}$ be the unique solution to \eqref{criticalxi} with $\xi$ replaced by $e_j$. Then we observe from the symmetry of $a(t,y)$ that, for a.e. $t\in I$,
\begin{align*}
&\tenchi a_{\rm hom}(t)e_j\cdot e_k
=
a_{\rm hom}(t)e_{k}\cdot e_j\displaybreak[0]\\
&=
\int_{\square}	a(t,y)(\nabla_y\Phi_{k}(t,y)+e_k)\cdot e_{j}\, dy\\
&\quad +
\underbrace{\int_{\square}	a(t,y)(\nabla_y\Phi_{k}(t,y)+e_k)\cdot \nabla_y\Phi_j(t,y)\, dy
+
\bigl\langle C_{\ast}t\partial_t\Phi_k(t),\Phi_j(t)\bigl\rangle_{H^1_{\rm per}(\square)/\R}}_{=0 \text{ by \eqref{CPcritical}}}\displaybreak[0]\\
&=
\int_{\square}	(\nabla_y\Phi_{k}(t,y)+e_k)\cdot \tenchi a(t,y)(\nabla_y\Phi_j(t,y)+e_j)\, dy
+
\bigl\langle C_{\ast}t\partial_t\Phi_k(t),\Phi_j(t)\bigl\rangle_{H^1_{\rm per}(\square)/\R}\displaybreak[0]\\
&=
a_{\rm hom}(t)e_j\cdot e_k\\
&\quad+
\left[\int_{\square}a(t,y)(\nabla_y\Phi_{j}(t,y)+e_j)\cdot\nabla_y\Phi_k(t,y)\, dy+\bigl\langle C_{\ast}t\partial_t\Phi_k(t), \Phi_j(t)\bigl\rangle_{H^1_{\rm per}(\square)/\R}\right].
\end{align*}
However, for the second term in the last line, it follows that 
\begin{align*}
\lefteqn{
\int_{\square}	a(t,y)(\nabla_y\Phi_{j}(t,y)+e_j)\cdot\nabla_y\Phi_k(t,y)\, dy+\bigl\langle C_{\ast}t\partial_t\Phi_k(t),\Phi_j(t)\bigl\rangle_{H^1_{\rm per}(\square)/\R}}\\
&\quad=
C_{\ast}t\Bigl(-\bigl\langle\partial_t\Phi_j(t),\Phi_k(t)\bigl\rangle_{H^1_{\rm per}(\square)/\R}+\bigl\langle\partial_t\Phi_k(t),\Phi_j(t)\bigl\rangle_{H^1_{\rm per}(\square)/\R}\Bigl)
\neq 0\quad \text{ for } j\neq k, \nonumber
\end{align*}
which completes the proof.

\begin{rmk}\label{skew}
\rm
The skew symmetric part of $a_{\rm hom}(t)$ is defined by
\begin{align*}
a_{\rm hom}^{\rm skew}(t)e_j\cdot e_k
&=
\left(\frac{a_{\rm hom}(t)-\tenchi a_{\rm hom}(t)}{2}\right)e_j\cdot e_k 
\\
&=
-\frac{1}{2}
C_{\ast}t\Bigl(-\bigl\langle\partial_t\Phi_j(t),\Phi_k(t)\bigl\rangle_{H^1_{\rm per}(\square)/\R}+\bigl\langle\partial_t\Phi_k(t),\Phi_j(t)\bigl\rangle_{H^1_{\rm per}(\square)/\R}\Bigl).\nonumber
\end{align*}
Then we note that
the skew-symmetric part of $a_{\rm hom}  (t)  $ makes no contribution to the divergence for a.e.~$t\in I$. Assume that $u_0$ and $\Phi_k$ are smooth enough. Then we find by the symmetry of the Hessian that, for a.e. $t\in I$,
\begin{align*}
\mathrm{div} (a_{\rm hom}^{\rm skew}(t) \nabla u_0)
&= \frac{1}{2}\mathrm{div} (a_{\rm hom}^{\rm skew}(t) \nabla u_0)\\
&\quad -
\frac{1}{4}\sum_{j,k=1}^N
\left[C_{\ast}t\int_{\square}\Bigl(-\partial_t\Phi_j(t,y)\Phi_k(t,y)+\partial_t\Phi_k(t,y)\Phi_j(t,y)\Bigl)\, dy
\right]\partial_{x_jx_k}^2u_{0}\\
&= \frac{1}{2}\mathrm{div} (a_{\rm hom}^{\rm skew}(t) \nabla u_0)\\
&\quad  + 
\frac{1}{4}\sum_{j,k=1}^N
\left[C_{\ast}t\int_{\square}\Bigl( -\partial_t\Phi_k(t,y)\Phi_j(t,y)+\partial_t\Phi_j(t,y)\Phi_k(t,y)  \Bigl)\, dy
\right]\partial_{ x_kx_j }^2u_{0}\\
&=
\frac{1}{2}\mathrm{div} (a_{\rm hom}^{\rm skew}(t) \nabla u_0)-\frac{1}{2}\mathrm{div} (a_{\rm hom}^{\rm skew}(t) \nabla u_0)=0,
\end{align*}
which yields the assertion. 
\end{rmk}

\section{Proof of a corrector result}
This section is devoted to proving Theorem \ref{CR} and Corollary \ref{CR2}.
\subsection{Proof of Theorem \ref{CR}}
Let $a_{\e}=a(\mx)$ for the sake of simplicity. 
To show Theorem \ref{CR}, we observe from \eqref{ellip} that
\begin{align*}
\lefteqn{\lambda\int_{0}^T\int_{\Omega}
\left|\nabla u_{\e}-\left(\nabla  u_0+\nabla_yu_1(x,t,\mx)\right)\right|^2\, dxdt}\\
&\stackrel{\eqref{ellip}}{\le}
\int_{0}^T\int_{\Omega}
a_{\e}\left(\nabla u_{\e}-\left(\nabla  u_0+\nabla_yu_1(x,t,\mx)\right)\right)\cdot\left(\nabla u_{\e}-\left(\nabla u_0+\nabla_yu_1(x,t,\mx)\right)\right)\, dxdt\nonumber\\
&=
\int_{0}^T\int_{\Omega}
a_{\e}\nabla u_{\e}\cdot \nabla u_{\e}\, dxdt
-2\int_{0}^T\int_{\Omega}	a_{\e}\nabla u_{\e}\cdot \left(\nabla  u_0+\nabla_yu_1(x,t,\mx)\right)\, dxdt\nonumber\\
&\quad	+
\int_{0}^T\int_{\Omega}
a_{\e}\left(\nabla  u_0+\nabla_yu_1(x,t,\mx)\right)\cdot \left(\nabla  u_0+\nabla_yu_1(x,t,\mx)\right)\, dxdt
=:I_1^\e-2I_2^\e+I_3^\e.\nonumber
\end{align*}
In what follows, we shall estimate these three terms, $I_1^\e$, $I_2^\e$ and $I_3^\e$ for all $r\in (0,  +\infty) $.

We first estimate $I_1^\e$. 
\begin{lem}\label{for I_1-2}
Under the same assumption as in Theorem \ref{CR}, it holds that
\begin{align*}
\lefteqn{	\lim_{\e\to 0_+}\int_0^T\int_{\Omega}a_\e\nabla u_\e(x,t)\cdot\nabla u_{\e}(x,t)\, dxdt}\\
&\quad=
\int_0^T\int_{\Omega}\int_\square a(y)\bigl(\nabla u_0(x,t)+\nabla_y u_1(x,t,y)\bigl)
\cdot \bigl(\nabla u_0(x,t)+\nabla_y u_1(x,t,y)\bigl)\, dydxdt.
\end{align*}
\end{lem}

\begin{proof}
From {\bf (A)}, \eqref{conv4} and (iii) of Lemma \ref{bdd2}, it follows that
\begin{align*}
\lefteqn{\int_0^T\int_{\Omega}a_\e\nabla u_\e(x,t)\cdot\nabla u_\e(x,t)\, dxdt}\\
&\stackrel{\eqref{weakform}}{=}
\int_0^T\int_{\Omega} f_\e(x,t)u_\e(x,t)\, dxdt -\int_0^T\int_{\Omega}\bigl(\partial_{tt}^2u_\e(x,t)+g(\mt)\partial_tu_\e(x,t)\bigl)u_\e(x,t)\, dxdt\\
&\stackrel{}{\to}
\int_0^T\int_{\Omega}f(x,t)u_0(x,t)\, dxdt -\int_0^T\int_{\Omega}w(x,t)u_0(x,t)\, dxdt\\
&=
\int_0^T\int_{\Omega}\int_\square a(y)\bigl(\nabla u_0(x,t)+\nabla_y u_1(x,t,y)\bigl)
\cdot \bigl(\nabla u_0(x,t)+\nabla_y u_1(x,t,y)\bigl)\, dydxdt 
\end{align*}
as $\e\to 0_+$. Here we used the fact that $\Phi_k$ is the unique solution to \eqref{CPslow} due to $a=a(y)$. This completes the proof.
\end{proof}

Before discussing the limit of $I_2^\e$, recall that $u_1$ is written by $u_1=\sum_{k=1}^N\partial_{x_k}u_0\Phi_k$. Due to the smoothness of $a(y)$,
noting that $a(y)(\nabla u_0+\nabla_yu_1)$ belongs to $[L^2(\Omega\times I;C_{\rm per}(\square))]^N$, 
we see by \cite[Theorem 4]{LNW} that it is an admissible test function in $[L^2(\Omega\times I\times \square)]^N$.
Hence \eqref{conv5.5} yields that
\begin{align}\label{I2}
I_2^\e\to
\int_0^T
\int_{\Omega}\int_{\square}
a(y)\bigl(\nabla u_0(x,t)+\nabla_y u_1(x,t,y)\bigl)
\cdot (\nabla u_0(x,t)+\nabla_y u_1(x,t,y))\, dydxdt
\end{align}
as $\e\to 0_+$.

We finally estimate $I_3^\e$. Thanks to $\nabla_y \Phi_k\in C_{\rm per}(\square)$, 
one can derive by Proposition \ref{mean} that
\begin{align}
I_3^\e	
&=	
\int_0^T\int_{\Omega}\Bigl[ a_{\e}\nabla u_{0}\cdot \nabla u_{0}
+2 a_{\e}\nabla_y u_1(x,t,\tfrac{x}{\e})\cdot\nabla u_0+a_{\e}\nabla_y u_1\cdot\nabla_y u_1(x,t,\tfrac{x}{\e})\Bigl]\, dxdt\label{I3}\\
&\to
\int_0^T\int_{\Omega}\int_{\square} \left[	a(y)\nabla u_{0}\cdot \nabla u_{0}
+2 a(y)\nabla_yu_1\cdot\nabla u_0 +a(y)\nabla_y u_1\cdot\nabla_y u_1  \right]\, dydxdt\nonumber\\
&=
\int_0^T\int_{\Omega}\int_{\square}	a(y)(\nabla u_0+\nabla_y u_1)\cdot (\nabla u_0+\nabla_y u_1)\, dydxdt\quad \text{ as }\ \e\to 0_+.	\nonumber
\end{align}
Consequently, with the aid of Lemma \ref{for I_1-2}, \eqref{I2} and \eqref{I3}, we obtain
$$\lim_{\e\to 0_+}(I_1^\e-2I_2^\e+I_3^\e)= 0,$$
which completes the proof.

\subsection{Proof of Corollary \ref{CR2}}
The strategy of the proof of Corollary \ref{CR2} is the same as Theorem \ref{CR} and it suffices to show the following
\begin{lem}
Under the same assumption as in Corollary \ref{CR2}, it holds that
\begin{align*}
\lefteqn{	\limsup_{\e\to 0_+}	\int_0^T \int_{\Omega} 
	a(t,\mx)\nabla u_\e(x,t)\cdot\nabla u_{\e}(x,t)\, dxdt} \\
&\quad \le
\int_0^T\int_{\Omega}\int_\square a(t,y)\bigl(\nabla u_0(x)+\nabla_y u_1(x,t,y)\bigl)
\cdot 
\bigl(\nabla u_0(x)+\nabla_y u_1(x,t,y)\bigl)\, dydxdt.\nonumber
\end{align*} 
\end{lem}
\begin{proof}
Let $a_{\e}=a(t,\mx)$ for simplicity.
Define $E^\e(u_\e(t))$ by
\begin{align*}
E^\e(u_\e(t))
&=\frac{1}{2}\|\partial_\rho u_\e(t)\|_{L^2(\Omega)}^2
+\frac{1}{2}\int_{\Omega} a_\e\nabla u_\e(x,t)\cdot \nabla u_\e(x,t)\, dx\\
&\quad+
\int_0^{t}g(\tfrac{\rho}{\e^r})\|\partial_\rho u_\e(\rho)\|_{L^2(\Omega)}^2\, d\rho
-\frac{1}{2}\int_0^{t}\int_\Omega \partial_\rho a(\rho,\mx)\nabla u_\e(x,\rho)\cdot \nabla u_\e(x,\rho)\, dxd\rho
\end{align*}
for $t\in I$. From (i), (ii) and (iii) of Lemma \ref{bdd} and {\bf (A)}, we have $|E^\e(u_\e(t))|\le C$ for all $t\in \overline{I}$. 
Moreover, we derive by Remark \ref{test reg} that,
for any $t\in I$ and $h\in (0,T-t)$,
\begin{align*}
\lefteqn{|E^\e(u_\e(t+h))-E^\e(u_\e(t))|}\\
&\le
\int_{t}^{t+h}
\Bigl| \langle \partial_{\rho\rho}^2u_{\e}(\rho),\partial_\rho u_\e(\rho)\rangle_{H^1_0(\Omega)}
+ \bigl(a_\e\nabla u_\e(\rho),\nabla\partial_\rho u_{\e}(\rho)\bigl)_{L^2(\Omega)}+g(\tfrac{\rho}{\e^r})\|\partial_\rho u_\e(\rho)\|^2_{L^2(\Omega)}\Bigl|\, d\rho\nonumber\\
&\stackrel{\eqref{weakform}}{\le}
\int_{t}^{t+h}\|f_{\e}(\rho)\|_{L^2(\Omega)}\|\partial_\rho u_\e(\rho)\|_{L^2(\Omega)} \, d\rho
\le
\|f_\e\|_{L^2(\Omega\times I)}\|\partial_\rho u_{\e}\|_{L^{\infty}(\overline{I};L^2(\Omega))}|h|^{1/2}, 
\end{align*}
which along with (ii) of Lemma \ref{bdd} and the boundedness of $(f_\e)$ in $L^2(\Omega\times I)$ yields the equicontinuous of $t\mapsto E^\e(u_\e(t))$ on $\overline{I}$. 
Then Ascoli-Arzel\'a's theorem ensures that there exists $\xi\in C(\overline{I})$ such that
\begin{align}\label{for I_1-1}
E^\e(u_\e)\to \xi\quad\text{ strongly in $C(\overline{I})$}.
\end{align}
Furthermore, it holds that, for any $t\in I$,
\begin{align*}
\lefteqn{E^\e(u_\e(t))}\\
&\quad=
\int_{0}^{t}
\Bigl| \langle \partial_{\rho\rho}^2u_{\e}(\rho),\partial_\rho u_\e(\rho)\rangle_{H^1_0(\Omega)}
+ \bigl(a_\e\nabla u_\e(\rho),\nabla\partial_\rho u_{\e}(\rho)\bigl)_{L^2(\Omega)}+g(\tfrac{\rho}{\e^r})\|\partial_\rho u_\e(\rho)\|^2_{L^2(\Omega)}\Bigl|\, d\rho\\
&\quad+
\frac{1}{2}\|v^1_\e\|_{L^2(\Omega)}^2+\frac{1}{2}\int_{\Omega}a(0,\mx)\nabla v^0_\e(x)\cdot \nabla v^0_\e(x)\, dx\nonumber\\
&\stackrel{\eqref{weakform}}{=}
\int_{0}^{t}\int_{\Omega}f_\e(x,\rho)\partial_{\rho}u_\e(x,\rho)\, dxd\rho
+\frac{1}{2}\|v^1_\e\|_{L^2(\Omega)}^2+\frac{1}{2}\int_{\Omega}a(0,\mx)\nabla v^0_\e(x)\cdot \nabla v^0_\e(x)\, dx.\nonumber
\end{align*}
Then, by \eqref{strong hinv}, \eqref{initialv-add} and \eqref{f-add}, $\xi(t)$ is identified with 
\begin{align*}
\xi(t)
&=
\frac{1}{2}\int_{\Omega}a_{\rm hom}(0)\nabla v^0(x)\cdot\nabla v^0(x)\, dx\\
&=
\frac{1}{2}\int_{\Omega}\int_\square a(0,y)\bigl(\nabla u_0(x)+\nabla_y u_1(x,0,y)\bigl)\cdot\nabla u_0(x)\, dydx.\nonumber
\end{align*}
Hence defining $J(t)$ by
$$
J(t):=\int_0^t\int_\Omega\int_\square \partial_\rho a(\rho,y)\bigl(\nabla u_0(x)+\nabla_y u_1(x,\rho,y)\bigl)\cdot \bigl(\nabla u_0(x)+\nabla_y u_1(x,\rho,y)\bigl)\, dydxd\rho
$$
and noting by \eqref{ellip2} and \cite[Example 1~in Section 7]{Zh} that 
\begin{align*}
\liminf_{\e\to 0_+}\int_0^t\int_\Omega (-\partial_\rho a_\e)\nabla u_\e(x,\rho )\cdot \nabla u_\e(x,\rho)\, dxd\rho \ge -J(t),
\end{align*}
one can derive by \eqref{for I_1-1} that 
\begin{align*}
\lefteqn{\limsup_{\e\to 0_+}\int_0^T\int_{\Omega}a_\e\nabla u_\e(x,t)\cdot \nabla u_\e(x,t)\, dxdt}
\\
&\le
\lim_{\e\to 0_+}\int_0^T 2E^\e(u_\e)(t)\, dt 
+
\limsup_{\e\to 0_+}\int_0^T\int_0^{t}\int_\Omega \partial_\rho a(\rho,\mx)\nabla u_\e(x,\rho)\cdot \nabla u_\e(x,\rho)\, dxd\rho dt\nonumber\\
&\le 
T\int_{\Omega}\int_\square a(0,y)\bigl(\nabla u_0(x)+\nabla_y u_1(x,0,y)\bigl)\cdot\nabla u_0(x)\, dydx+\int_0^TJ(t)\, dt.\nonumber
\end{align*}
We estimate the second term below. Since $u_1(x,\cdot,\cdot)$ is smooth in $\square\times (\eta,T)$ for all $\eta\in I$ due to the smoothness of $a(t,y)$, it holds that
\begin{align*}
&\int_\eta^T
\int_\eta^{t}\int_\Omega\int_\square \partial_\rho a(\rho,y)\bigl(\nabla u_0(x)+\nabla_y u_1(x,\rho,y)\bigl)\cdot \bigl(\nabla u_0(x)+\nabla_yu_1(x,\rho,y)\bigl)\, dydxd\rho dt\\
&=
\int_\eta^T\int_\Omega\int_\square a(t,y)\bigl(\nabla u_0(x)+\nabla_y u_1(x,t,y)\bigl)\cdot \bigl(\nabla u_0(x)+\nabla_y u_1(x,t,y)\bigl)\, dydxdt\\
&\quad -
(T-\eta)\int_\Omega\int_\square a(\eta,y)\bigl(\nabla u_0(x)+\nabla_y u_1(x,\eta,y)\bigl)\cdot 
\bigl(\nabla u_0(x)+\nabla_y u_1(x,\eta,y)\bigl)\, dydx\\
&\quad -
2\int_\eta^T\int_\eta^{t}\int_\Omega\int_\square a(\rho,y)\bigl(\nabla u_0(x)+\nabla_y u_1(x,\rho,y)\bigl)\cdot \nabla_y\partial_\rho u_1(x,\rho,y)\, dydxd\rho dt\\
&=:J_1^\eta+J_2^\eta+J_3^\eta.
\end{align*}
Then one can verify that
\begin{equation*}
J_1^\eta\to 
\int_0^T\int_\Omega\int_\square a(t,y)\bigl(\nabla u_0(x)+\nabla_y u_1(x,t,y)\bigl)\cdot \bigl(\nabla u_0(x)+\nabla_y u_1(x,t,y)\bigl)\, dydxdt
\end{equation*}
and 
\begin{equation}\label{J2est}
J_2^\eta\to -
T\int_\Omega\int_\square a(0,y)\bigl(\nabla u_0(x)+\nabla_y u_1(x,0,y)\bigl)\cdot 
 \nabla u_0(x)\, dydx
\end{equation}
as $\eta\to 0_+$. Here we used the fact \eqref{CPcritical} at $t=0$ in \eqref{J2est}. 
Furthermore, if $r\neq 2$, $J_3^\eta=0$ readily follows due to $u_1=u_1(x,y)$.
On the other hand, if $r=2$, \eqref{CPcritical} and \eqref{ellip3} yield that
\begin{align}
J_3^\eta
&\stackrel{\eqref{ellip3}}{=}
2\int_\eta^T\int_\eta^{t}\int_\Omega\int_\square \partial_\rho a(\rho,y)\bigl(\nabla u_0(x)+\nabla_y u_1(x,\rho,y)\bigl)\cdot \nabla_y\partial_\rho u_1(x,\rho,y)\, dydxd\rho dt\nonumber\\
&\le
2\int_\eta^T\int_\Omega\int_\square a(t,y)\bigl(\nabla u_0(x)+\nabla_y u_1(x,t,y)\bigl)\cdot \nabla_y\partial_\rho u_1(x,t,y)\, dydxdt\nonumber\\
&\quad -
2(T-\eta)\int_\Omega\int_\square a(\eta,y)\bigl(\nabla u_0(x)+\nabla_y u_1(x,\eta,y)\bigl)\cdot \nabla_y\partial_\rho u_1(x,\eta,y)\, dydx\nonumber\\
&\quad -
2\int_\eta^T\int_\eta^{t}\int_\Omega\int_\square a(\rho,y)\bigl(\nabla u_0(x)+\nabla_y u_1(x,\rho,y)\bigl)\cdot \nabla_y\partial_{\rho\rho}^2 u_1(x,\rho,y)\, dydxd\rho dt\nonumber\\
&\stackrel{\eqref{CPcritical}}{=}
-2\int_\eta^T C_\ast t \|\partial_\rho u_1(t)\|_{L^2(\Omega\times \square)}^2\, dt
+
2(T-\eta)C_\ast \eta\|\partial_\rho u_1(\eta)\|_{L^2(\Omega\times \square)}^2\nonumber\\
&\quad +
\int_\eta^T\int_\eta^{t} C_\ast \rho \frac{d}{d\rho}\|\partial_\rho u_1(\rho)\|_{L^2(\Omega\times \square)}^2\, d\rho dt\nonumber\\
&\le
-\int_\eta^T C_\ast t \|\partial_\rho u_1(t)\|_{L^2(\Omega\times \square)}^2\, dt
+
2(T-\eta)C_\ast \eta\|\partial_\rho u_1(\eta)\|_{L^2(\Omega\times \square)}^2
\le 
0 \nonumber
\end{align}
as $\eta\to 0_+$. Hence we conclude that 
\begin{align*}
\int_0^TJ(t)\, dt
&= \lim_{\eta\to 0_+}J_1^\eta+\lim_{\eta\to 0_+}J_2^\eta
+
\lim_{\eta\to 0_+}J_3^\eta
\\
&\le
\int_0^T\int_\Omega\int_\square a(t,y)\bigl(\nabla u_0(x)+\nabla_y u_1(x,t,y)\bigl)\cdot \bigl(\nabla u_0(x)+\nabla_y u_1(x,t,y)\bigl)\, dydxdt\\
&\quad -
T\int_\Omega\int_\square a(0,y)\bigl(\nabla u_0(x)+\nabla_y u_1(x,0,y)\bigl)\cdot 
 \nabla u_0(x)\, dydx,
\end{align*}
which completes the proof.
\end{proof}

\section*{Acknowledgment}
The author is partially supported by Division for Interdisciplinary Advanced Research and Education, Tohoku University and Grant-in-Aid for JSPS Fellows (No.~20J10143). He would like to thank Professor Goro Akagi (Tohoku University) who is his supervisor, for many stimulating discussions.

\end{document}